\documentclass[12pt]{article}
\usepackage{amsmath}
\usepackage{amssymb}
\usepackage{theorem}

\numberwithin{equation}{section}

\newtheorem{thm}{Theorem}[section]
\newtheorem{lemma}[thm]{Lemma}
\newtheorem{prop}[thm]{Proposition}
\newtheorem{cor}[thm]{Corollary}
{\theorembodyfont{\rmfamily}
\newtheorem{defn}[thm]{Definition}
\newtheorem{ex}[thm]{Example}
\newtheorem{rmk}[thm]{Remark}
}

\newcommand{\qed}{\hfill \mbox{\raggedright \rule{.07in}{.1in}}}
 
\newenvironment{proof}{\vspace{1ex}\noindent{\bf
Proof}\hspace{0.5em}}{\hfill\qed\vspace{1ex}}
\newenvironment{pfof}[1]{\vspace{1ex}\noindent{\bf Proof of
#1}\hspace{0.5em}}{\hfill\qed\vspace{1ex}}

\newcommand{\R}{{\mathbb R}}
\newcommand{\C}{{\mathbb C}}
\newcommand{\Z}{{\mathbb Z}}
\newcommand{\D}{{\mathbb D}}
\newcommand{\N}{{\mathbb N}}
\renewcommand{\P}{{\mathbb P}}

\newcommand{\Lip}{\operatorname{Lip}}

\renewcommand{\Re}{\operatorname{Re}}
\renewcommand{\Im}{\operatorname{Im}}

\newcommand{\SMALL}{\textstyle}
 \newcommand{\BIG}{\displaystyle}

\title{Operator renewal theory and mixing rates 
for dynamical systems with infinite measure}

 \author{Ian Melbourne \thanks{Department of Mathematics, University of Surrey,
Guildford, Surrey GU2 7XH, UK}
 \and
 Dalia Terhesiu \thanks{
Department of Mathematics, University of Surrey,
Guildford, Surrey GU2 7XH, UK}
}

\date{24 August, 2010.   Revised 10 August, 2011. \\
Minor corrections 5 March, 2015.}

\begin{document}

\maketitle

 \begin{abstract}
We develop a theory of operator renewal sequences in the
context of infinite ergodic theory.  For large classes of dynamical
systems preserving an infinite measure, we determine the asymptotic behaviour
of iterates $L^n$ of the transfer operator.
This was previously  an intractable problem.

Examples of systems covered by our results include (i) parabolic
rational maps of the complex plane and (ii) (not necessarily
Markovian) nonuniformly expanding interval maps with indifferent
fixed points.

In addition, we give a particularly simple proof of pointwise dual ergodicity 
(asymptotic behaviour of $\sum_{j=1}^nL^j$) for the class of systems under
consideration.

In certain situations, including Pomeau-Manneville intermittency maps, we 
obtain higher order expansions for $L^n$ and rates of mixing.
Also, we obtain error estimates in the associated Dynkin-Lamperti arcsine laws.

\vspace{1ex}
This version includes minor corrections in Sections 10 and 11, and corresponding modifications of certain statements in Section 1.  All main results are unaffected.  In particular, Sections 2--9 are unchanged from the published version.
 \end{abstract}


\paragraph{AMS Classification codes:}    
37A40,   37A25,   37A50,   60K05

\newpage

\section{Introduction} 
\label{sec-intro}

In finite ergodic theory, much recent interest has focussed on the statistical
properties of smooth dynamical systems with strong hyperbolicity
(expansion/contraction) properties.   Landmark results include the proof of
exponential decay of correlations for certain classes of uniformly
hyperbolic flows~\cite{Chernov98,Dolgopyat98,Liverani04} and
planar dispersing 
billiards~\cite{Young98}.  The latter is part of a general 
scheme~\cite{Young98,Young99}
for estimating decay of correlations, or mixing rates, for discrete time
dynamical systems.   

For systems with subexponential decay of correlations, most approaches yielded
only  upper bounds for mixing rates.  
Sarig~\cite{Sarig02} introduced a powerful new technique, {\em operator
renewal theory}, to obtain precise asymptotics and hence sharp mixing rates.
This is an extension of scalar renewal theory from probability theory
to the operator situation.  The technique was substantially
extended and refined by Gou\"ezel~\cite{Gouezel04, Gouezel05}.   

Garsia \& Lamperti~\cite{GarsiaLamperti62} developed a theory of renewal
sequences with infinite mean in the probabilistic setting.
The techniques are very different from the finite mean case, and draw heavily 
on the theory of regular 
variation~\cite{BinghamGoldieTeugels,Feller66,Karamata33}.   
A natural question is to develop renewal operator theory in the infinite
mean case.  

\vspace{-2ex}
\paragraph{Renewal sequences}
In probability theory, renewal sequences relate return probabilities to
a specified ``nice'' set with first return probabilities.    The analogue in
ergodic theory arises in the study of first return maps.

Let $(X,\mu)$ be a measure space (finite or infinite), and 
$f:X\to X$ a conservative 
measure preserving map.   Fix $Y\subset X$ with $\mu(Y)>0$.
Let $\varphi:Y\to\Z^+$ be the first return time 
$\varphi(y)=\inf\{n\ge1:f^ny\in Y\}$ (finite almost everywhere by 
conservativity).  Let $L:L^1(X)\to L^1(X)$ denote the transfer operator 
(Perron-Frobenius operator) for $f$ and define 
\[
T_nv=1_YL^n(1_Yv),\enspace n\ge0, \qquad R_nv=1_YL^n(1_{\{\varphi=n\}}v),\enspace n\ge1.
\]
Note that $T_n$ and $R_n$ can be viewed as operators on $L^1(Y)$ with $T_0=I$.
Thus $T_n$ corresponds to returns to $Y$ and $R_n$ corresponds to first 
returns to $Y$.   The relationship $T_n=\sum_{j=1}^n T_{n-j}R_j$
generalises the classical notion of renewal sequences in probability theory.

In the infinite mean setting of Garsia \& Lamperti~\cite{GarsiaLamperti62},
a crucial requirement is that the first return probabilities have regularly
varying tails.   In our setting it is natural to assume that the return time 
probabilities have regularly varying tails.
Indeed, many of the results in the
infinite ergodic theory literature rely crucially on such an 
assumption~\cite{Aaronson81,Thaler98,ThalerZweimuller06,Zweimuller00}.
Under this assumption, and certain functional-analytic hypotheses on the
operators $R_n$, we obtain detailed results 
(Theorems~\ref{thm-conv},~\ref{thm-bound},~\ref{thm-liminf})
on the asymptotic behaviour
of the operators $T_n$ as $n\to\infty$.   
This has strong ramifications for the asymptotics of the iterates
$L^n$, and hence for the underlying dynamical system.

\vspace{-2ex}
\paragraph{Maps with indifferent fixed points}
An important class of examples is provided by interval
maps with indifferent fixed
points, in particular  the Pomeau-Manneville intermittency
maps~\cite{PomeauManneville80} which are uniformly expanding away from
an indifferent fixed point at $0$.   
To fix notation, we focus on the version studied
by Liverani~{\em et al.}~\cite{LiveraniSaussolVaienti99}:
\begin{align} \label{eq-LSV}
fx=\begin{cases} x(1+2^\alpha x^\alpha), & 0<x<\frac12 \\ 2x-1 & \frac12<x<1
\end{cases}.
\end{align}

When $\alpha=0$, this is the doubling map, and Lebesgue measure is 
invariant, ergodic and exponentially mixing.  For $\alpha\in(0,1)$,
there is still a unique ergodic invariant probability measure $\mu$ 
absolutely continuous with respect to Lebesgue measure, but the rate of
mixing is polynomial: 
$\int_X v\,w\circ f^n \,d\mu-\int_X v\,d\mu \int_Xw\,d\mu \le 
C_{v,w}n^{-(\beta-1)}$
where $\beta=\frac{1}{\alpha}$ for all $w\in L^\infty(X)$ and all
$v$ sufficiently regular (for example, H\"older continuous).  
Hu~\cite{Hu04} proved that this decay rate is optimal;
a special case of the theory of~\cite{Gouezel04,Sarig02}.

For $\alpha\ge1$, we are in the situation of infinite ergodic theory.
There no longer exists an absolutely continuous invariant probability
measure, but there is a unique (up to scaling)
$\sigma$-finite, absolutely continuous invariant measure $\mu$.   
Previous studies established {\em pointwise dual ergodicity}: 
$a_n^{-1}\sum_{j=1}^nL^jv \to {\rm const}\,\int_X v\,d\mu$ almost
everywhere for all 
$v\in L^1(X)$, where $a_n=n^{\beta}$ for $\beta\in(0,1)$ and
$a_n=n/\log n$ when $\beta=1$.   

An important refinement is the limit theorem of
Thaler~\cite{Thaler95} where the convergence of $a_n^{-1}\sum_{j=1}^nL^jv$
is shown to be uniform on compact subsets of $(0,1]$
for all observables of the form $v=u/h$ where $u$ is Riemann integrable
and $h$ is the density.  

The results of~\cite{Thaler95} are formulated for Markov maps of the interval
with indifferent fixed points.
Zweim\"uller~\cite{Zweimuller98,Zweimuller00} relaxed the Markov
condition and systematically studied a large class of
non-Markovian nonuniformly
expanding interval maps,  called \emph{AFN maps}.  (See Section~\ref{sec-AFN}
for a precise definition of AFN map.)  In 
particular,~\cite{Zweimuller98} obtained a spectral decomposition into
 basic (conservative and ergodic) sets and proved that for each basic set there
is a $\sigma$-finite absolutely continuous invariant measure, unique up to
scaling.  
The results in~\cite{Thaler95} on uniform pointwise dual
ergodicity were extended in~\cite{Zweimuller00} to the class of AFN maps.

Understanding the asymptotics of $L^n$, rather than $\sum_{j=1}^nL^j$,
turns out to be a much more difficult problem, even for~\eqref{eq-LSV}.
Previously, the sole success in this direction was obtained by 
Thaler~\cite{Thaler00}.
However, the class of systems covered by~\cite{Thaler00} is quite restrictive
and includes the family~\eqref{eq-LSV} only for $\beta=1$.   

It is this situation that we have sought to redress in this paper.
It is convenient to describe our main results in the setting of AFN maps
$f:X\to X$,
though our general theory goes much further, as described later on.
Let $X'\subset X$ denote the complement of the indifferent fixed
points.   For any compact subset $A\subset X'$,
the construction in~\cite{Zweimuller98} yields 
a suitable first return set $Y$
containing $A$.   Fix such a set $Y$ with first return time function
$\varphi:Y\to\Z^+$.   Then we assume that the tail
probabilities are regularly varying:
$\mu(\varphi>n)=\ell(n)n^{-\beta}$ where $\beta\in(0,1]$ and $\ell(x)$ is slowly varying
($\ell:(0,\infty)\to(0,\infty)$ is measurable and $\lim_{x\to\infty}\ell(\lambda x)/\ell(x)=1$ for all $\lambda>0$).
(For~\eqref{eq-LSV}, $\ell$ is asymptotically constant
and $\beta=\frac{1}{\alpha}$.)

Now we can state our results for AFN maps.
Set $d_\beta ={\SMALL\frac{1}{\pi}}\sin\beta\pi$ for $\beta\in(0,1)$ and 
$d_1=1$.  Define $m(x)=\ell(x)$ for $\beta\in(0,1)$
and $m(x)= \sum_{j=1}^{[x]}\ell(j)j^{-1}$ for $\beta=1$.

\begin{thm} \label{thm-AFN}
Suppose that $f:X\to X$ is a topologically mixing AFN map with $\sigma$-finite
absolutely continuous invariant measure $\mu$ (with density $h$) and regularly
varying tail probabilities.
\begin{itemize}
\item[(a)] If $\beta\in(\frac12,1]$, then 
$\lim_{n\to\infty} m(n)n^{1-\beta}L^n v=d_\beta\int_X v\,d\mu$
uniformly on compact subsets of $X'$ 
for all $v:X\to\R$ of the form $v=u/h$ where $u$ is Riemann integrable on $X$.

\item[(b)] If $\beta\in(0,1]$, then for $v:X\to\R$ that is $\mu$-integrable and Riemann integrable, there is a subset $E\subset\N$ of zero
density such that
$\lim_{n\to\infty,\;n\not\in E} m(n)n^{1-\beta}L^n v=d_\beta\int_X v\,d\mu$
pointwise on $X'$.

Moreover, if $v\ge0$, then 
$\liminf_{n\to\infty} m(n)n^{1-\beta}L^n v=d_\beta\int_X v\,d\mu$
pointwise on $X'$.

\item[(c)] If $\beta\in(0,\frac12)$, then $L^nv=O(\ell(n)n^{-\beta})$ uniformly
on compact subsets of $X'$
for all measurable $v:X\to\R$ such that $vh$ is bounded.
\end{itemize}
\end{thm}

\begin{rmk}  
(i)  It is known that the asymptotic behaviour of $L^n$ might be
complicated.  Chung~\cite[Section I.10]{Chung} gives an example of a null recurrent Markov chain for which the ratio of $n$-step transition probabilities $p_{ij}^n/p_{k\ell}^{(n)}$ has no limit as $n\to\infty$ (the regular variation assumption on the return time probabilities is violated).   
Hajian \& Kakutani~\cite{HajianKakutani64} (see also~\cite[Proposition~1.4.7]{Aaronson}) prove that
 there always exist {\em weakly wandering} sets $W$ of positive measure.   For such sets, $\int_WL^n1_W\,d\mu=0$ for infinitely many $n$.
 (Such indicator functions $1_W$ do not lie in our class of observables $v$).

\vspace{1ex}
\noindent(ii)
In the special case of the family~\eqref{eq-LSV},
Theorem~\ref{thm-AFN}(a) recovers the result of Thaler~\cite{Thaler00} for
$\beta=1$ and the cases $\beta<1$ are new.
Parts~(b,c) are probably not optimal for~\eqref{eq-LSV} but are the best 
one can expect in the general setting, see Remark~\ref{rmk-Gouezel}.

 \vspace{1ex}
 \noindent (iii)  In addition to yielding convergence results for $L^n$
 (rather than $\sum_{j=1}^n L^j$) our methods also cover much wider 
 classes of observables than was previously possible.   An indicative example
 is the family~\eqref{eq-LSV} where $\beta\in(\frac12,1]$.    
 There is a constant $C>1$ such that $C^{-1}x^{-\frac{1}{\beta}}\le h(x)
 \le Cx^{-\frac{1}{\beta}}$.
 Consider observables of the form $v(x)=x^q$.
 Whereas the results of~\cite{Thaler95,Zweimuller00} yield uniform convergence
 of $\sum_{j=1}^n L^jv$ (and of $L^nv$ when $\beta=1$) on compact subsets 
 of $(0,1]$ if and only if $\beta q\ge1$, our results apply for $\beta(q+1)>1$ 
(see Theorem~\ref{thm-PMsecond_q}).
 
\vspace{1ex}
\noindent(iv) An immediate consequence of Theorem~\eqref{thm-AFN}(a) is that 
if $\beta\in(\frac12,1]$, $v$ is of the required form $v=u/h$, and 
$w\in L^1(X)$ is supported on a compact subset of $X'$, then 
$\lim_{n\to\infty}m(n)\int_X v\,w\circ f^n\,d\mu=
d_\beta\int_Xv\,d\mu\int_X w\,d\mu$.   

\vspace{1ex}
\noindent(v)  
The situation changes considerably if $\int_X v\,d\mu=0$.
We have the following result which has no counterpart in standard renewal 
theory, though Gou\"ezel~\cite{Gouezel04} proves an analogous (and equally 
unexpected) result in the case of finite ergodic theory:
\end{rmk}

\begin{thm} \label{thm-zero}
Suppose that $f:X\to X$ is a topologically mixing AFN map with regularly
varying tail probabilities, $\beta\in(0,1)$.
Suppose that $v$ is of bounded variation and is supported on a compact subset
of $X'$.   If $\int_X v\,d\mu=0$, then
$L^nv=O(\ell(n)n^{-\beta})$ uniformly on compact subsets of $X'$.
\end{thm}

\vspace{-3ex}
\paragraph{Second order asymptotics and rates of mixing}
In certain situations, including the family~\eqref{eq-LSV}, 
the tail probabilities
satisfy $\mu(\varphi>n)=cn^{-\beta}+O(n^{-q})$ for some $q>1$, $c>0$.   
It is then possible to obtain higher order expansions of $L^nv$ on compact
sets for bounded variation observables $v$ supported
on a compact subset of $(0,1]$.    For example, in the 
specific case of~\eqref{eq-LSV}, $\beta\in(\frac12,1)$, we prove that 
$|n^{1-\beta}\int_0^1 v\,w\circ f^n\,d\mu-\int_0^1 v\,d\mu\int_0^1 w\,d\mu|
\le Cn^{-\gamma}$ where $\gamma=\min\{1-\beta,\beta-\frac12\}$.
The rate of mixing is sharp for $\beta\ge\frac34$, and 
we obtain precise second order asymptotics provided $\beta>\frac34$. 

These and related results such as error rates in
the Dynkin-Lamperti arcsine laws~\cite{Dynkin61,Lamperti58} are discussed in 
Section~\ref{sec-second}.

\vspace{1ex}
The remainder of the paper is organised as follows.
In Section~\ref{sec-ORT}, we describe the general framework
for our results on the renewal operators $T_n$.
Sections~\ref{sec-T},~\ref{sec-Fourier} and~\ref{sec-half}
contain the proofs for $\frac12<\beta<1$.
In Section~\ref{sec-one}, we cover the case $\beta=1$.
In Section~\ref{sec-pde}, we give a particularly elementary, self-contained,
proof of pointwise dual ergodicity for all $\beta$.
In Section~\ref{sec-zerohalf}, we prove our results for 
$0<\beta\le\frac12$.
In Section~\ref{sec-second}, we formulate and prove results on 
higher order asymptotics.
In Section~\ref{sec-L}, we show how to pass from $T_nv=1_YL^n(1_Yv)$ to $L^n$.
Finally in Section~\ref{sec-examples}, we show that our theory applies to 
large classes of examples including AFN maps (in particular, we prove 
Theorem~\ref{thm-AFN}) and systems for which the
first return map is Gibbs-Markov.   The latter includes parabolic rational
maps of the complex plane~\cite{ADU93}.

\vspace{-2ex}
\paragraph{Notation}
We use ``big O'' and $\ll$ notation interchangeably, writing
$A_n=O(a_n)$ or $A_n\ll a_n$ as $n\to\infty$ if there is a constant
$C>0$ such that $\|A_n\|\le Ca_n$ for all $n\ge1$ (for $A_n$ operators
and $a_n\ge0$ scalars).    We write $A_n\sim c_nA$
as $n\to\infty$ if $\lim_{n\to\infty}\|A_n/c_n-A\|=0$ (for $A_n,A$ operators
and $c_n>0$ scalars).

\section{General framework}
\label{sec-ORT}

Let $(X,\mu)$ be an infinite measure space, and $f:X\to X$ a conservative 
measure preserving map.   Fix $Y\subset X$ with $\mu(Y)\in(0,\infty)$ and
rescale $\mu$ so that $\mu(Y)=1$.
Let $\varphi:Y\to\Z^+$ be the first return time $\varphi(y)=\inf\{n\ge1:f^ny\in Y\}$
and define the first return map $F=f^\varphi:Y\to Y$.

The return time function $\varphi:Y\to\Z^+$ satisfies 
$\int_Y \varphi\,d\mu=\infty$.
Our crucial assumption is that the tail probabilities are regularly varying:
\begin{itemize}
\item[]
$\mu(y\in Y:\varphi(y)>n)=\ell(n)n^{-\beta}$ where $\ell$ is slowly varying and
$\beta\in(0,1]$.
\end{itemize}

Recall that the transfer operator $R:L^1(Y)\to L^1(Y)$ for the first return 
map $F:Y\to Y$ is defined via the formula
$\int_Y Rv\,w\,d\mu = \int_Y v\,w\circ F\,d\mu$, $w\in L^\infty(Y)$. 
Let $\D=\{z\in\C:|z|<1\}$ and
$\bar\D=\{z\in\C:|z|\le1\}$.
Given $z\in\bar\D$, we define $R(z):L^1(Y)\to L^1(Y)$ to be
the operator $R(z)v=R(z^\varphi v)$.   Also, for each $n\ge1$, we define
$R_n:L^1(Y)\to L^1(Y)$, $R_nv=R(1_{\{\varphi=n\}}v)$.
It is easily verified that $R(z)=\sum_{n=1}^\infty R_nz^n$.

Our assumptions on the first return map $F:Y\to Y$ are functional-analytic in
nature.
We assume that there is a function space $\mathcal{B}\subset L^\infty(Y)$
containing constant functions,
with norm $\|\;\|$ satisfying $|v|_\infty\le \|v\|$ for $v\in\mathcal{B}$,
such that for some constant $C>0$:
\begin{itemize}
\item[(H1)] For all $n\ge1$,
$R_n:\mathcal{B}\to\mathcal{B}$ is a bounded linear 
operator with $\|R_n\|\le C\mu(\varphi=n)$.
\end{itemize}

In particular, $z\mapsto R(z)$ is a continuous family of bounded linear 
operators on $\mathcal{B}$ for $z\in\bar\D$.   
Since $R(1)=R$ and $\mathcal{B}$ contains constant functions, 
$1$ is an eigenvalue of $R(1)$.  We require:

\begin{itemize}
\item[(H2)] \begin{itemize}
\item[(i)] The eigenvalue $1$ is simple and isolated
in the spectrum of $R(1)$.
\item[(ii)] For $z\in\bar\D\setminus\{1\}$, the spectrum of $R(z)$ does
not contain $1$.
\end{itemize}
\end{itemize}
Denote the spectral projection corresponding to the simple eigenvalue $1$
for $R(1)$ by $P$.  Then $(Pv)(y)\equiv\int_Y v\,d\mu$.

\subsection{Asymptotics of $T_n$}
\label{sec-conv}

We state our main results for the operators $T_n$.
Since $T_nv=1_YL^n(1_Yv)$, we obtain
precise results for the convergence of $L^nv$ on $Y$
for observables $v$ supported on $Y$.   The restriction to $Y$ is lifted
in Section~\ref{sec-L}.
Throughout, we assume regularly varying tails $\mu(\varphi>n)=\ell(n)n^{-\beta}$
and hypotheses (H1) and (H2).
Set $d_\beta ={\SMALL\frac{1}{\pi}}\sin\beta\pi$ for $\beta\in(0,1)$ and 
$d_1=1$.  Define $m(n)=\ell(n)$ for $\beta\in(0,1)$
and $m(n)= \sum_{j=1}^n\ell(j)j^{-1}$ for $\beta=1$.

In some of our statements, the observable $v$ is not mentioned.   Here,
we are speaking of convergence of operators on the Banach space
$\mathcal{B}$.
So for example, Theorem~\ref{thm-conv} states that 
$\sup_{v\in\mathcal{B},\;\|v\|=1}\|m(n)n^{1-\beta}T_nv-d_\beta\int_Y v\, d\mu\|\to0$ as $n\to\infty$.
Since $\mathcal{B}$ is embedded in $L^\infty(Y)$, this immediately
implies almost sure convergence (at a uniform rate) on $Y$.    
Redefining sequences on a set of measure zero, we obtain uniform convergence on $Y$.    For brevity, we will speak of uniform convergence throughout this paper without further comment.

\begin{thm}   \label{thm-conv}
If $\beta\in(\frac12,1]$,  then
$\lim_{n\to\infty}m(n) n^{1-\beta}T_n =d_\beta P$.
\end{thm}

The next result gives upper bounds on the decay rate of $T_n$
for $\beta\le\frac12$, and an improved upper bound for $\beta\ge\frac12$
when the observable is of mean zero.

\begin{thm} \label{thm-bound}
\begin{itemize}
\item[(a)]
If $\beta=\frac12$, then
$T_n \ll \ell(n) n^{-\frac12}\int_{1/n}^\pi \ell(1/\theta)^{-2}\theta^{-1}\,
d\theta$.
\item[(b)] If $\beta\in(0,\frac12)$, then
$T_n \ll \ell(n) n^{-\beta}$.
\item[(c)] 
If $\beta\in(0,1)$, and $v\in\mathcal{B}$ satisfies $\int_Y v\,d\mu=0$,
then $T_nv \ll \ell(n) n^{-\beta}$.
\end{itemize}
\end{thm}

As indicated in~\cite{GarsiaLamperti62}, the estimate in 
Theorem~\ref{thm-bound}(b) is essentially optimal.
However, we recover certain aspects of Theorem~\ref{thm-conv}
even for $\beta\le\frac12$.
Recall that $E\subset\N$ has {\em density zero} if 
$\lim_{n\to\infty} \frac1n\sum_{j=1}^n 1_E(j)=0$.

\begin{thm}   \label{thm-liminf}
Let $\beta\in(0,\frac12]$ and $v\in\mathcal{B}$. 
\begin{itemize}
\item[(a)] For all $y\in Y$, there exists a set $E$ of zero density such that 
$\lim_{n\to\infty,\;n\not\in E} \ell(n)n^{1-\beta}(T_nv)(y) = 
d_\beta \int_Y v\,d\mu$.
\item[(b)]  If $v\ge0$, then
$\liminf_{n\to\infty} \ell(n)n^{1-\beta}T_nv = d_\beta \int_Y v\,d\mu$ 
pointwise on $Y$.
\end{itemize}
\end{thm}

\begin{rmk} \label{rmk-Gouezel}    
In general, Theorem~\ref{thm-conv} fails for $\beta\le\frac12$.
However, there is the possibility of proving the result for all $\beta$
under additional hypotheses.   Indeed,
Gou\"ezel~\cite{GouezelPC} informs us that he is able to prove 
Theorem~\ref{thm-conv} for all $\beta\in(0,1)$ under the additional
assumption that $\mu(\varphi=n)=O(\ell(n)n^{-(\beta+1)}$).
In particular, Gou\"ezel's result applies to the family~\eqref{eq-LSV}.

It is worth recalling the situation from the scalar case
(where the $T_n$ are probabilities instead of operators).
Under the additional assumption $\mu(\varphi=n)=O(\ell(n)n^{-(\beta+1)})$,
Garsia \& Lamperti~\cite{GarsiaLamperti62} were able to extend 
Theorem~\ref{thm-conv} to the case $\beta\in(\frac14,\frac12)$.   
This is the only part of~\cite{GarsiaLamperti62} that we are unable to
generalise to the operator setting.
However, an argument of Doney~\cite{Doney97} applies to all $\beta\in(0,1)$
and according to
Gou\"ezel~\cite{GouezelPC} this argument can be extended to the operator
case.
\end{rmk}

\begin{rmk} \label{rmk-DK}
An immediate consequence of Theorem~\ref{thm-conv} is that $Y$ 
is a Darling-Kac set whenever $\beta>\frac12$.    
We refer to Aaronson~\cite[Chapter~3]{Aaronson} and Aaronson, Denker \& Urbanksi~\cite[Section~1]{ADU93} for 
definitions and numerous consequences.   Other consequences include the
Dynkin-Lamperti arcsine laws, see Thaler~\cite{Thaler98}.   
Indeed our main result significantly simplifies the derivation of the 
arcsine laws, see~\cite{Thaler00}.  
\end{rmk}

\subsection{Preliminaries}

For convenience, we state Karamata's Theorem on the integration
of regularly varying sequences~\cite{BinghamGoldieTeugels, Feller66}.

\begin{prop} \label{prop-Karamata}   
Suppose that $\ell$ is slowly varying.
\begin{itemize}
\item[(a)] 
If $p>-1$, then
$\sum_{j=1}^n\ell(j)j^p\sim \frac{1}{p+1}\ell(n)n^{p+1}$ as $n\to\infty$.
\item[(b)] The function $\tilde\ell(x)
=\sum_{j=1}^{[x]}\ell(j)j^{-1}$   is slowly varying and
$\lim_{n\to\infty}\ell(n)/\tilde\ell(n)=0$.~\qed
\end{itemize}
\end{prop}
The following consequence of (H1) and regularly varying tails is standard.

\begin{prop} \label{prop-H1}
There is a constant $C>0$ such that
$\|R(\rho e^{i(\theta+h)})-R(\rho e^{i\theta})\|\le C m(1/h)h^\beta$ 
and $\|R(\rho)-R(1)\|\le C m(\frac{1}{1-\rho})(1-\rho)^\beta$ for
all $\theta\in[0,2\pi)$, $\rho\in(0,1]$, $h>0$.
\end{prop}

\begin{proof}   We sketch the calculation for the first estimate.    By
Proposition~\ref{prop-Karamata},
\[
R(\rho e^{i(\theta+h)})-R(\rho e^{i\theta}) 
\ll h\sum_{j=1}^Nj\mu(\varphi=j)+\sum_{j>N}\mu(\varphi=j)
\ll hm(N)N^{1-\beta}+m(N)N^{-\beta},
\]
so the result follows with $N=[h^{-1}]$.
\end{proof}

By (H2), there exists $\epsilon>0$ such that $R(z)$ has a continuous family of
simple eigenvalues $\lambda(z)$ for $z\in \bar\D\cap B_\epsilon(1)$ with
$\lambda(1)=1$.  Let $P(z):\mathcal{B}\to\mathcal{B}$ denote the 
corresponding family of spectral
projections, with complementary projections $Q(z)=I-P(z)$ and $P(1)=P$.
Also, let $v(z)\in\mathcal{B}$ denote the corresponding family of 
eigenfunctions normalised so that $\int_Y v(z)\,d\mu=1$ for all $z$.   
In particular, $v(1)\equiv1$.

\begin{cor}    \label{cor-H1}
The estimates for $R(z)$ in Proposition~\ref{prop-H1}
are inherited by the families $P(z)$, $Q(z)$, $\lambda(z)$ and  $v(z)$,
where defined.  
\end{cor}

\begin{proof}
This is a standard consequence of perturbation theory for analytic families of
operators~\cite{Kato}.
\end{proof}

We have defined the bounded linear operators $T_n,R_n:\mathcal{B}\to\mathcal{B}$
given by
\[
T_nv=1_YL^n(1_Yv),\enspace n\ge0, \qquad
R_nv=1_YL^n(1_{\{\varphi=n\}}v)=R(1_{\{\varphi=n\}}v),\enspace n\ge1.
\]
The power series
\[
T(z)=\sum_{n=0}^\infty T_n z^n,\enspace z\in\D,\qquad
R(z)=\sum_{n=1}^\infty R_n z^n,\enspace z\in\bar\D,
\]
are analytic on the open unit disk $\D$, and $R(z)$ is continuous on $\bar\D$
by (H1).  We have the usual relation
$T_n=\sum_{j=1}^n T_{n-j}R_j$ for $n\ge1$,
and it follows that $T(z)=I+T(z)R(z)$ on $\D$.  Hence
$T(z)=(I-R(z))^{-1}$ on $\D$.   
By (H2)(ii), $T(z)$ extends continuously to $\bar\D-\{1\}$ via
the formula $T(z)=(I-R(z))^{-1}$.

\begin{prop} \label{prop-PQ}
There exists $\epsilon,C>0$ such that 
$\|T(z)-(1-\lambda(z))^{-1}P(z)\|\le C$ for $z\in \bar\D\cap B_\epsilon(1)$,
$z\neq1$, and $\|T(z)\|\le C$ for $z\in\bar\D\setminus B_\epsilon(1)$.
\end{prop}

\begin{proof}
Choose $\epsilon>0$ so that the family of simple eigenvalues 
$\lambda(z)$ is defined on $\bar\D\cap B_\epsilon(1)$.
For $z\in \bar\D\cap B_\epsilon(1)$, we can write
$R(z)=\lambda(z)P(z)+R(z)Q(z)$.
If in addition $z\neq1$, then we have (in an obvious notation)
\[
T(z)=(1-\lambda(z))^{-1}P(z)+(I-R(z))^{-1}Q(z). 
\]
By (H2), the second term is uniformly bounded in the operator norm,
and $T(z)$ is uniformly bounded for $z\in\bar\D\setminus B_\epsilon(1)$.
\end{proof}

\subsection{Strategy of the proof of Theorem~\ref{thm-conv}}
\label{sec-strategy}

Our aim is to compute the operators $T_n$ defined above.
Most of the analysis is carried out on the unit circle $S^1$, so it
is convenient to abuse notation, writing $T(\theta)$ instead of 
$T(e^{i\theta})$ and so on.    For $\beta\in(\frac12,1)$,
our  treatment follows Garsia \& 
Lamperti~\cite{GarsiaLamperti62} but there are some significant differences
in two of the three steps.

The first step, Section~\ref{sec-T}, is to study the singularity for $T(\theta)$
at $\theta=0$.   The argument 
in~\cite{GarsiaLamperti62} is scalar and similar results can be found 
in~\cite{Feller49,Hardy31}.   In our situation, the key is to 
use the fact that the return time $\varphi$ lies in the domain of a 
stable law, and a Nagaev-type argument due to
Aaronson \& Denker~\cite{AaronsonDenker01} shows that 
$(1-\lambda(\theta))^{-1}\sim {\rm const}\, \ell(1/\theta)^{-1}\theta^{-\beta}$ 
when $\beta\in(0,1)$.  By Corollary~\ref{cor-H1} and Proposition~\ref{prop-PQ},
$T(\theta)\sim {\rm const}\, \ell(1/\theta)^{-1}\theta^{-\beta}P$.

In particular $T(\theta)\in L^1$ with Fourier coefficients $\widehat T_n$.
The second step is to relate $T_n$ to $\widehat T_n$.
In the scalar case,~\cite{GarsiaLamperti62} invokes ideas of
Herglotz~\cite{Herglotz11} on analytic functions with positive real part.
A different approach is required here since we are working with operators.  
In Section~\ref{sec-Fourier}, we verify that $T_n=\widehat T_n$. 

In the final step, Section~\ref{sec-half}, we 
investigate the behaviour of $\widehat T_n$ as $n\to\infty$, directly
following~\cite{GarsiaLamperti62}.

\subsection{Tower extensions}
\label{sec-tower}

The following tower construction will be required in Sections~\ref{sec-pde}
and~\ref{sec-L}. 

Starting from the first return map 
$F=f^\varphi:Y\to Y$, there is a standard way of constructing an
extension $f_\Delta:\Delta\to\Delta$ of the underlying map $f:X\to X$.
Define the {\em tower} $\Delta=\{(y,j)\in Y\times\Z:0\le j<\varphi(y)\}$
and the {\em tower map} $f_\Delta:\Delta\to\Delta$ given by
$f_\Delta(y,j)=(y,j+1)$ for $j\le \varphi(y)-2$
and $f_\Delta(y,j)=(Fy,0)$ for $j= \varphi(y)-1$.

The base of the tower $\{(y,0):y\in Y\}$ is naturally identified with
$Y$ and so we may regard $Y$ as a subset of both $X$ and $\Delta$.
Let $\mu_\Delta$ be the unique $f_\Delta$-invariant measure on $\Delta$ 
that agrees with the underlying measure $\mu$ on the common subset $Y$.

Define the projection $\pi:\Delta\to X$, $\pi(y,j)=f^jy$.
Then $\pi f_\Delta = f \pi$ and $\pi_*\mu_\Delta=\mu$.
Thus $f_\Delta$ is an extension of $f$ with the same first return map 
$F:Y\to Y$ and return time function $\varphi:Y\to\Z^+$ as the original map.

\section{Asymptotics of $T(\theta)$}
\label{sec-T}

In this section, we obtain an asymptotic expression for $T(\theta)$ as 
$\theta\to0^+$. Throughout, $\beta\in(0,1)$.
The main part of the analysis is to understand the asymptotics of the 
leading eigenvalue $\lambda(\theta)$.   In certain situations,
we obtain a higher order expansion.
Let $c_\beta=-i\int_0^\infty e^{i\sigma}\sigma^{-\beta}\,d\sigma$.   

\begin{lemma} \label{lem-AD}  Let $\beta\in(0,1)$.  
As $\theta\to0^+$,
\begin{itemize}
\item[(a)]
$\lambda(\theta)=1-c_\beta\ell(1/\theta)\theta^\beta(1+o(1))$.
\item[(b)] $T(\theta)-(1-\lambda(\theta))^{-1}P=O(1)$.
\item[(c)] 
$T(\theta)=c_\beta^{-1}\ell(1/\theta)^{-1}\theta^{-\beta}(1+o(1))P+O(1)$.
\end{itemize}
\end{lemma}

\begin{proof}
(a) This is part of~\cite[Theorem~5.1]{AaronsonDenker01}. 
(The main ideas of the proof are reproduced later in the proof of
Lemma~\ref{lem-ADerror} and Lemma~\ref{lem-zestimate}.)

\noindent (b) By Proposition~\ref{prop-PQ},
\[
T(\theta)= (1-\lambda(\theta))^{-1}P(\theta)+
O(1)
=(1-\lambda(\theta))^{-1}P+
(1-\lambda(\theta))^{-1}(P(\theta)-P)+O(1).
\]
By Corollary~\ref{cor-H1}, $P(\theta)-P\ll \ell(1/\theta)\theta^\beta$,
so the result follows from (a).

\noindent (c) is immediate by (a) and (b).
\end{proof}

The following expansion for $\lambda(\theta)$ will be used in 
proving results on second order asymptotics (Section~\ref{sec-second}).

\begin{lemma} \label{lem-ADerror}
Suppose that $\mu(\varphi>n)=c(n^{-\beta}+H(n))$, where $c>0$ 
and $H(n)=O(n^{-q})$, $q>1$.
Let $H_1(x)=[x]^{-\beta}-x^{-\beta}+H([x])$ and set
$c_H=\int_0^\infty H_1(x)\,dx$.    Then
\[
\lambda(\theta)=1-cc_\beta\theta^\beta+icc_H\theta+O(\theta^{2\beta},\theta^q),
\enspace\text{as $\theta\to0^+$}.
\]
\end{lemma}

\begin{proof}
We follow~\cite[Theorem~5.1]{AaronsonDenker01}.
Recall that $v(\theta)$ is the eigenfunction corresponding to $\lambda(\theta)$
normalised so that $\int v(\theta)\,d\mu=1$.
Since $v(0)\equiv1$,
it follows from Corollary~\ref{cor-H1} that $|v(\theta)-1|_\infty\ll \theta^\beta$.
Let $\mathcal{F}_0$ denote the $\sigma$-algebra generated by $\varphi$.
Define the step function 
$\hat v(\theta):[0,\infty)\to\C$ given by
$\hat v(\theta)\circ \varphi=E(v(\theta)|\mathcal{F}_0)$ and the distribution function
$G(x)=\mu(\varphi\le x)$.
Then
\begin{align*}
\lambda(\theta) & =
{\SMALL\int}_Y R(\theta)v(\theta)\,d\mu={\SMALL\int}_Y e^{i\theta\varphi}v(\theta) \,d\mu 
=1+{\SMALL\int}_Y (e^{i\theta\varphi}-1)v(\theta)\,d\mu \\[.75ex] & =
1+{\SMALL\int}_0^\infty (e^{i\theta x}-1)\hat v_\theta(x)\,dG(x),
\end{align*}
where $\sup_{x\ge0}|\hat v_\theta(x)-1|\ll \theta^\beta$ and
$1-G(x)=c(x^{-\beta}+H_1(x)).
$
Here $H_1(x)=O(x^{-q'})$ as $x\to\infty$, 
where $q'=\min\{q,\beta+1\}>1$.   
In particular, $H_1$ is integrable.

Write $\hat v_\theta=1+v_\theta^1-v_\theta^2+iv_\theta^3-iv_\theta^4$,
where $v_\theta^s\ge0$ and $\sup_x |v_\theta^s(x)|\ll \theta^\beta$ for
$s=1,2,3,4$.   Define the positive measure $dG_\theta^s=v_\theta^s\, dG$.
Then
\begin{align*}
\lambda(\theta)&=1+\int_0^\infty (e^{i\theta x}-1)\,dG(x)+
\sum_{s=1}^4 q_s\int_0^\infty (e^{i\theta x}-1)\,dG_\theta^s(x) \\
&=1+i\theta\int_0^\infty e^{i\theta x}(1-G(x))\,dx+
\sum_{s=1}^4 q_si\theta\int_0^\infty e^{i\theta x}g_\theta^s(x)(1-G(x))\,dx,
\end{align*}
where $q_1=1$, $q_2=-1$, $q_3=i$, $q_4=-i$, 
$g_\theta^s(x)=\int_x^\infty v_\theta^s(u)\,dG(u)/
\int_x^\infty \,dG(u)\ll\theta^\beta$.

Now, 
\begin{align*}
& i\theta \int_0^\infty e^{i\theta x}(1-G(x))\,dx 
 = 
ic\theta\int_0^\infty e^{i\theta x}x^{-\beta}\,dx +
ic\theta\int_0^\infty e^{i\theta x}H_1(x)\,dx 
\\ & \qquad = ic\theta^\beta \int_0^\infty e^{i\sigma}\sigma^{-\beta}\,d\sigma +
ic\theta\int_0^\infty H_1(x)\,dx +
ic\theta\int_0^\infty (e^{i\theta x}-1)H_1(x)\,dx \\
& \qquad = -cc_\beta\theta^\beta +icc_H \theta + ic\theta A
\end{align*}
where
\begin{align*}
A & = 
\int_0^\infty  (e^{i\theta x}-1)H_1(x)\,dx =
\int_0^{1/\theta} (e^{i\theta x}-1)H_1(x)\,dx +
\int_{1/\theta}^\infty (e^{i\theta x}-1)H_1(x)\,dx \\
& \ll \theta\int_0^{1/\theta} xH_1(x)\,dx +
\int_{1/\theta}^\infty H_1(x)\,dx 
\ll \theta\int_0^{1/\theta} x^{1-q'}\,dx +
\int_{1/\theta}^\infty x^{-q'}\,dx 
 \ll \theta^{q'-1}.
\end{align*}

It remains to estimate the terms 
$\int_0^\infty e^{i\theta x}g_\theta^s(x)(1-G(x))\,dx$.
We give the details for $\int_0^\infty \sin\theta x\,g_\theta^s(x)(1-G(x))\,dx$; the case with $\sin$ replaced by $\cos$ is identical.
Since $x\mapsto g_\theta^s(x)(1-G(x))$ is decreasing for each fixed $\theta$,
we can write
\begin{align*}
 \int_0^\infty \sin\theta x\,g_\theta^s(x)(1-G(x))\,dx  
& \le
\int_0^{\pi/\theta} \sin\theta x\,g_\theta^s(x)(1-G(x))\,dx 
\\ & \ll \theta^\beta\int_0^{\pi/\theta} x^{-\beta}\,dx \ll \theta^{2\beta-1},
\end{align*}
giving the required upper bound, and the lower bound is obtained in the
same way.~
\end{proof}

\section{Identification of the Fourier coefficients}
\label{sec-Fourier}

Let $\widehat R_n$ and $\widehat T_n$ denote the Fourier coefficients of $R(\theta)$ and $T(\theta)$.
By (H1), $R$ is uniformly absolutely summable on $S^1$.   Therefore 
$\widehat R_n=R_n$.   
In this section, we verify that $\widehat T_n=T_n$ for all $n\ge0$.
Throughout, $\beta\in(0,1)$.

\begin{lemma}   \label{lem-zestimate}
There exists $\epsilon,C>0$ such that
$|1-\lambda(z)|\ge C\ell(1/\theta)\theta^\beta$ for all 
$z=\rho e^{i\theta}\in\bar\D \cap B_\epsilon(1)$.
\end{lemma}

\begin{proof}
We start off by mimicking the proof of Lemma~\ref{lem-ADerror}.
Consider functions $v_z:[0,\infty)\to[0,\infty)$ satisfying either (i) $v_z\equiv1$ or
(ii) $|v_z|_\infty=o(1)$ as $z\to1$.
Write $z=e^{-u+i\theta}$, $0\le u,\theta<\epsilon$.
Then the expansion $\lambda(z)-1$ leads to a linear combination
of five integrals of the form
\begin{align} \label{eq-five}
I=\int_0^\infty (e^{(-u+i\theta)x}-1)v_z(x)\,dG(x)=
(-u+i\theta)\int_0^\infty e^{(-u+i\theta)x}g_z(x)(1-G(x))dx,
\end{align}
where $g_z\ge0$ and either (i) $g_z\equiv1$ or (ii) $|g_z|_\infty=o(1)$ as $z\to1$.
Moreover there is one integral of type (i) and we show that
this satisfies the desired lower bound, whilst the four integrals
of type (ii) are negligible.

Recall that $1-G(x)=\mu(\varphi>x)=\ell([x])[x]^{-\beta}
=x^{-\beta}h(x)$ where $h(x)=\ell(x)(1+o(1))$.   Substituting
$\sigma=\theta x$,
\[
I=(-u+i\theta)\ell(1/\theta)
\theta^{\beta-1}\int_0^\infty e^{-\sigma y}e^{i\sigma}g_z(\sigma/\theta)
\sigma^{-\beta}h(\sigma/\theta)\ell(1/\theta)^{-1}\,d\sigma,
\]
where $y=u/\theta$.
We estimate the oscillatory integrals in the same way that alternating series are
estimated in Leibnitz's theorem, making extensive use of the fact that
$\sigma\mapsto e^{-\sigma y}g_z(\sigma/\theta)
\sigma^{-\beta}h(\sigma/\theta)$ is decreasing for each fixed $\theta,y$.

Write
$|I|= (u^2+\theta^2)^{1/2}\ell(1/\theta)\theta^{\beta-1}A$,
where in cases (i) and (ii) respectively,
\begin{align*}
A_{(i)}& \ge \int_0^{3\pi/2} \cos\sigma\, e^{-\sigma y}\sigma^{-\beta}h(\sigma/\theta)\ell(1/\theta)^{-1}\,d\sigma, \\
A_{(ii)} & \le 
 \Bigl(\int_0^{\pi/2}\cos\sigma+\int_0^{\pi}\sin\sigma\Bigr)e^{-\sigma y}
\sigma^{-\beta}
g_z(\sigma/\theta)h(\sigma/\theta)\ell(1/\theta)^{-1}\,d\sigma.
\end{align*}

We divide the region $y>0$ into the regions $y\in(0,1/\delta]$
and $y\ge1/\delta$ where $\delta$ is chosen sufficiently small.
We have the Potter's bounds~\cite{Potter42},~\cite[Theorem 1.5.6]{BinghamGoldieTeugels}: $C^{-1}\sigma^\beta\le h(\sigma/\theta)\ell(1/\theta)^{-1}\le C\sigma^{-1}$ uniformly in $\theta>0$, $\sigma\in(0,2\pi]$, where $C$ is a constant.
Hence
\begin{align*}
A_{(i)} & \ge
  \frac{\sqrt 2}{2}\int_0^{\pi/4}e^{-\sigma y}\sigma^{-\beta}h(\sigma/\theta)\ell(1/\theta)^{-1} \,d\sigma+O\Bigl(
\int_{\pi/4}^{3\pi/2}e^{-\sigma y} \,d\sigma\Bigr) 
 \\ & = \frac{\sqrt 2}{2}\int_0^{\pi/4}e^{-\sigma y}\sigma^{-\beta}h(\sigma/\theta)\ell(1/\theta)^{-1} \,d\sigma+O(y^{-1}e^{-(\pi/4) y}),
\end{align*}
\begin{align*}
A_{(ii)} & \le 2|g_z|_\infty\int_0^{\pi/4}e^{-\sigma y}\sigma^{-\beta}h(\sigma/\theta)\ell(1/\theta)^{-1} \,d\sigma +O\Bigl(\int_{\pi/4}^{\pi}e^{-\sigma y} \,d\sigma\Bigr) 
\\ & \ll |g_z|_\infty A_{(i)} + O(y^{-1}e^{-(\pi/4) y}).
\end{align*}
Since $|g_z|_\infty=o(1)$, we can choose $\epsilon$ sufficiently small
that 
\[
|1-\lambda(z)|\gg u\ell(1/\theta)\theta^{\beta-1}
\Big\{\int_0^{\pi/4}e^{-\sigma y}\sigma^{-\beta}h(\sigma/\theta)\ell(1/\theta)^{-1} \,d\sigma +O(y^{-1}e^{-(\pi/4) y})\Bigr\}.
\]
Furthermore,
\[
\int_0^{\pi/4}e^{-\sigma y}\sigma^{-\beta}h(\sigma/\theta)\ell(1/\theta)^{-1} \,d\sigma \gg \int_0^{\pi/4}e^{-\sigma y}\,d\sigma=y^{-1}+O(y^{-1}e^{-(\pi/4) y}).
\]
For $y\ge 1/\delta$ with $\delta$ sufficiently small, the terms 
$O(y^{-1}e^{-(\pi/4) y})$ are negligible so that
$1-\lambda(z)\gg u\ell(1/\theta)\theta^{\beta-1}y^{-1}=\ell(1/\theta)\theta^\beta$.

It remains to consider the complementary region $y\in(0,1/\delta]$.
Note that $\int_0^{3\pi/2} e^{-\sigma y}\cos\sigma\, \sigma^{-\beta}\,d\sigma$
depends continuously on $y$ and is positive for all $y\ge0$. It 
follows by compactness that 
$\int_0^{3\pi/2} e^{-\sigma y} \cos\sigma\, \sigma^{-\beta}\,d\sigma$ is bounded
away from zero for $y\in[0,1/\delta]$.   Moreover there exists $b\in(0,3\pi/2)$
such that 
$\int_b^{3\pi/2} e^{-\sigma y} \cos\sigma\, \sigma^{-\beta}\,d\sigma$ is bounded
away from zero for $y\in[0,1/\delta]$.   
By uniform convergence of slowly varying functions~\cite[Theorem~1.2.1]{BinghamGoldieTeugels}, we can shrink $\epsilon$ if necessary so that
$|h(\sigma/\theta)/\ell(1/\theta)-1|$ is as small as desired, uniformly
in $\sigma\in[b,3\pi/2]$ and $\theta\in(0,\epsilon]$.
Hence $A_{(i)}\ge \int_0^{3\pi/2} \cos\sigma\, e^{-\sigma y}\sigma^{-\beta}h(\sigma/\theta)\ell(1/\theta)^{-1}\,d\sigma$ which is bounded away from zero for $y\in(0,1/\delta]$.

Also, by Potter's bounds for any $\beta'\in(\beta,1)$,
\begin{align*}
A_{(ii)} & 
 \le 2|g_z|_\infty \int_0^{\pi}e^{-\sigma y}
\sigma^{-\beta}h(\sigma/\theta)\ell(1/\theta)^{-1}\,d\sigma 
 \ll |g_z|_\infty \int_0^{\pi}\sigma^{-\beta'}\,d\sigma\ll |g_z|_\infty=o(1).
\end{align*}
Hence $A_{(ii)}$ is negligible relative to $A_{(i)}$ and we obtain
$1-\lambda(z)\gg 
\theta\ell(1/\theta)\theta^{\beta-1}=\ell(1/\theta)\theta^\beta$
completing the proof.
\end{proof}

It is convenient in the next result (and crucial in Section~\ref{sec-one}) to discuss
the real part of an operator.
We recall that the operators $T_n$ are defined on the real Banach space 
$\mathcal{B}$.   Passing to the complexification, there is a natural conjugation
$u+iv\mapsto u-iv$ on $\mathcal{B}$.   Given an operator $A:\mathcal{B}\to\mathcal{B}$, define the conjugate $\bar A:\mathcal{B}\to\mathcal{B}$ by setting
$\bar A v = \overline{A\bar v}$, and the real part $\Re A=\frac12(A+\bar A)$.
In the case of the operator $T(z)=\sum_{n=0}^\infty T_nz^n$, this coincides
with the definitions $\overline{T(z)}=T(\bar z)$ and $\Re T(z)=\sum_{n=0}^\infty T_n\Re(z^n)$.

\begin{cor}   \label{cor-Fourier}
$T_n=\widehat T_n=
\frac{1}{\pi}\Re\int_0^\pi T(e^{i\theta})e^{-in\theta}\,d\theta$ 
for all $n\ge0$.
\end{cor}

\begin{proof}
Since $T(z)=\sum_{j=0}^\infty T_jz^j$ is analytic on the open unit disk $\D$,
$T_n =\frac{1}{2\pi}\rho^{-n}\int_0^{2\pi}T(\rho e^{i\theta}) 
e^{-in\theta}\,d\theta$, for all $\rho\in(0,1)$.

By Proposition~\ref{prop-PQ}, $T(z)=O(1)$ on $\bar\D\setminus B_\epsilon(1)$.
Further, on $\bar\D\cap B_\epsilon(1)$,
$T(z)=(1-\lambda(z))^{-1}P(z)+O(1)\ll (1-\lambda(z))^{-1}+O(1)$.
By Lemma~\ref{lem-zestimate}, $\|T(\rho e^{i\theta})\|\ll 
\ell(1/\theta)^{-1}\theta^{-\beta}$ for $z=\rho e^{i\theta}\in
\bar\D$ uniformly in $\rho$.
Since $\ell(1/\theta)^{-1}\theta^{-\beta}$ is integrable, it follows from the 
dominated convergence theorem (as $\rho\to1$) that
$T_n= \frac{1}{2\pi}\int_0^{2\pi}T(e^{i\theta}) e^{-in\theta}\,d\theta=
\widehat T_n$.   Since $T(\bar z)=\overline{T(z)}$, we obtain the expression 
$\frac{1}{\pi}\Re\int_0^\pi T(e^{i\theta})e^{-in\theta}\,d\theta$.
\end{proof}

\section{Convergence for $\beta\in(\frac12,1)$}
\label{sec-half}

In this section, we prove Theorem~\ref{thm-conv} for $\beta\in(\frac12,1)$.

\begin{lemma} \label{lem-GL}
Let $\beta\in(\frac12,1)$.  Let $n\ge1$, $a\in[1,n]$.  
Then for any $\beta'\in(0,\beta)$, 
\[
\ell(n) n^{1-\beta}
\int_{a/n}^\pi T(\theta) e^{-in\theta}\,d\theta\ll a^{-(2\beta'-1)}.
\]
If $\ell$ is asymptotically constant, then the result holds with $\beta'=\beta$.
\end{lemma}

\begin{proof}
By Lemma~\ref{lem-AD}(c), we have the estimate $\|T(\theta)\|\ll 
\ell(1/\theta)^{-1}\theta^{-\beta}$.   The proof uses this fact together
with Proposition~\ref{prop-H1}, and follows
Garsia \& Lamperti~\cite[p.~231]{GarsiaLamperti62}.  We give the details partly
for completeness and partly because we want to make explicit certain 
estimates that will be used in Section~\ref{sec-second}.

First, write
\[
I=\int_{a/n}^\pi T(\theta) e^{-in\theta}\,d\theta
=-\int_{(a+\pi)/n}^{\pi+\pi/n}T(\theta-\pi/n) e^{-in\theta}\,d\theta,
\]
so 
\[
2I=\int_{a/n}^\pi T(\theta) e^{-in\theta}\,d\theta
-\int_{(a+\pi)/n}^{\pi+\pi/n}T(\theta-\pi/n) e^{-in\theta}\,d\theta=I_1+I_2+I_3,
\]
where
\begin{align*}
I_1  & = \int_\pi^{\pi+\pi/n}T(\theta-\pi/n)e^{-in\theta}\,d\theta, \qquad
I_2  = \int_{a/n}^{(a+\pi)/n}T(\theta-\pi/n)e^{-in\theta}\,d\theta, \\
I_3 & =\int_{(a+\pi)/n}^\pi \{ T(\theta)-T(\theta-\pi/n)\}e^{-in\theta}\,d\theta.
\end{align*}
Clearly, $I_1\ll 1/n$, while
\begin{align*}
I_2 & \ll \int_{a/n}^{(a+\pi)/n}\ell(1/\theta)^{-1}\theta^{-\beta}\,d\theta \ll
\ell(n)^{-1}n^{-(1-\beta)}\int_a^{a+\pi}[\ell(n)/\ell(n/\sigma)]\sigma^{-\beta}
\,d\sigma \\ & 
\ll \ell(n)^{-1}n^{-(1-\beta)}\int_a^{a+\pi}\sigma^{-\beta'} \,d\sigma 
=\ell(n)^{-1}n^{-(1-\beta)}a^{1-\beta'}\{(1+\pi/a)^{1-\beta'}-1\} \\ &
\ll \ell(n)^{-1}n^{-(1-\beta)}a^{-\beta'}.
\end{align*}
By the resolvent identity and Proposition~\ref{prop-H1} (with $m(x)=\ell(x)$),
\begin{align*}
I_3 & \ll \int_{(a+\pi)/n}^\pi \|T(\theta)\|\|T(\theta-\pi/n)\|\|R(\theta)-R(\theta-\pi/n)\|\,d\theta \\
& \ll \ell(n/\pi)n^{-\beta}\int_{(a+\pi)/n}^\pi \ell(1/\theta)^{-1}\ell(1/(\theta-\pi/n))^{-1}\theta^{-\beta} (\theta-\pi/n)^{-\beta}\,d\theta \\
& = \ell(n/\pi)n^{-\beta}\int_{a/n}^{\pi-\pi/n} \ell(1/(\theta+\pi/n))^{-1}\ell(1/\theta)^{-1} (\theta+\pi/n)^{-\beta}\theta^{-\beta}\,d\theta.
\end{align*}
By Potter's bounds, $\ell(1/(\theta+\pi/n))^{-1}\ll\ell(1/\theta)^{-1}$ for 
$n\theta\ge 1$.  Hence,
\begin{align*}
I_3 & \ll \ell(n)n^{-\beta}\int_{a/n}^\pi \ell(1/\theta)^{-2}\theta^{-2\beta} \,d\theta =
\ell(n)^{-1}n^{-(1-\beta)}\int_{a}^{n\pi} [\ell(n)/\ell(n/\sigma)]^2\sigma^{-2\beta} \,d\sigma
\\
& \ll \ell(n)^{-1}n^{-(1-\beta)}\int_{a}^{n\pi} \sigma^{-2\beta'} \,d\sigma
\ll \ell(n)^{-1}n^{-(1-\beta)}a^{-(2\beta'-1)}.
\end{align*}
Altogether, we obtain
$\ell(n)n^{1-\beta}I \ll n^{-\beta}+a^{-\beta'}+a^{-(2\beta'-1)}\ll
a^{-(2\beta'-1)}$ as required.
\end{proof}

\begin{lemma} \label{lem-GLeigen}
Let $\beta\in(0,1)$.  Let $n\ge1$, $a\in(0,\epsilon n)$.  Then
\[
\lim_{a\to\infty}\lim_{n\to\infty}\ell(n) 
n^{1-\beta}\int_0^{a/n}(1-\lambda(\theta))^{-1} e^{-in\theta}\,d\theta = 
d_\beta',
\]
where $d_\beta'=i{\SMALL\int}_0^\infty e^{-i\sigma}\sigma^{-\beta}\,d\sigma/{\SMALL\int}_0^\infty e^{i\sigma}\sigma^{-\beta}\,d\sigma$.
\end{lemma}

\begin{proof}
This is identical to~\cite[Lemma~3.4.1]{GarsiaLamperti62} and we give
the proof only for completeness.
By Lemma~\ref{lem-AD}, we can write 
$(1-\lambda(\theta))^{-1}=c_\beta^{-1}\ell(1/\theta)^{-1}\theta^{-\beta}h(\theta)$
where $c_\beta=-i{\SMALL\int}_0^\infty e^{i\sigma}\sigma^{-\beta}\,d\sigma$ and
$\lim_{\theta\to0} h(\theta)=1$.
Hence
\begin{align*}
\int_0^{a/n}(1-\lambda(\theta))^{-1} e^{-in\theta}\,d\theta
 & =n^{-1}\int_0^a (1-\lambda(\sigma/n))^{-1}e^{-i\sigma}\,d\sigma \\ &
=c_\beta^{-1}n^{-(1-\beta)}\int_0^a e^{-i\sigma}\sigma^{-\beta}\ell(n/\sigma)^{-1}h(\sigma/n)\,d\sigma, 
\end{align*}
so that
\[
\ell(n)n^{1-\beta}\int_0^{a/n}(1-\lambda(\theta))^{-1} e^{-in\theta}\,d\theta 
=
c_\beta^{-1}\int_0^a e^{-i\sigma}\sigma^{-\beta}[\ell(n/\ell(n/\sigma)]h(\sigma/n)\,d\sigma.
\]
For fixed $a$, it follows from the dominated convergence theorem that
\[
\lim_{n\to\infty}
\ell(n)n^{1-\beta}\int_0^{a/n}(1-\lambda(\theta))^{-1} e^{-in\theta}\,d\theta 
=c_\beta^{-1}\int_0^a e^{-i\sigma}\sigma^{-\beta}\,d\sigma,
\]
and the result follows.
\end{proof}

\begin{pfof}{Theorem~\ref{thm-conv}, $\beta\in(\frac12,1)$.}
By Section~\ref{sec-Fourier}, 
$T_n=\frac{1}{\pi}\Re \int_0^\pi T(\theta)e^{-in\theta}\,d\theta$.
Let
\begin{align*}
& D(a,n) =
\int_0^\pi T(\theta)e^{-in\theta}\,d\theta-
\int_0^{a/n}(1-\lambda(\theta))^{-1}Pe^{-in\theta}\,d\theta \\
& \qquad  = \int_0^{a/n}\bigl\{T(\theta)-(1-\lambda(\theta))^{-1}P\bigr\}
e^{-in\theta}\,d\theta+
\int_{a/n}^\pi T(\theta)Pe^{-in\theta}\,d\theta,
\end{align*}
so $D(a,n)\ll a/n+\ell(n)^{-1}n^{-(1-\beta)}a^{-(2\beta'-1)}$ by 
Lemma~\ref{lem-AD}(b) and Lemma~\ref{lem-GL}.
Hence $\lim_{a\to\infty}\lim_{n\to\infty}\ell(n)n^{1-\beta}D(a,n)=0$.
By Lemma~\ref{lem-GLeigen},
$\lim_{a\to\infty}\lim_{n\to\infty}\ell(n)n^{1-\beta}T_n=
\frac{1}{\pi}\Re d_\beta'=d_\beta$.
The result follows since $T_n$ is independent of $a$.  
\end{pfof}

\section{Convergence for $\beta=1$}
\label{sec-one}

In this section, we prove Theorem~\ref{thm-conv} in the case $\beta=1$.
There are several differences from the case $\beta\in(\frac12,1)$.
First, $T(e^{i\theta})\not\in L^1$; instead it is shown below that
$\Re T(e^{i\theta})$ is integrable.

Estimating $\Re\{(1-\lambda(z))^{-1}\}$ is slightly easier than in 
Section~\ref{sec-Fourier} but estimating $\Re T(z)$ is harder since
$\Re\{(1-\lambda(z))^{-1}(P(z)-P)\}$ is not dominated by 
$\Re\{(1-\lambda(z))^{-1}\}$.  
As a consequence, $\Re T(z)=\Re T(\rho e^{i\theta})$ is not dominated by a 
single integrable function of $\theta$, see Lemma~\ref{lem-R1} below.

We have $\mu(\varphi>n)=\ell(n)n^{-1}$, where $\ell$ is slowly varying
and \mbox{$\sum \ell(n)n^{-1}=\infty$}.
Let $\tilde\ell(x)=m(x)=\sum_{j=1}^{[x]} \ell(j)j^{-1}$.
Then $\tilde\ell$ is monotone increasing and 
$\lim_{n\to\infty}\tilde\ell(n)=\infty$.
By Proposition~\ref{prop-Karamata}(b), $\tilde\ell$ is slowly varying and
$\ell(n)/\tilde\ell(n)\to0$ as $n\to\infty$.
Up to asymptotic equivalence, we have the alternative definitions
$\tilde\ell(x)=\int_1^x  \ell(y)y^{-1}\,dy$ and
$\tilde\ell(x)=\int_0^x  (1-G(y))\,dy$ where $G(x)=\mu(\varphi\le x)$.

\begin{prop} \label{prop-elltilde}
$\BIG\int_0^{1/y} \frac{\ell(1/\theta)}{\theta(\tilde\ell(1/\theta))^2}\,d\theta=
\int_y^\infty \frac{\ell(x)}{x(\tilde\ell(x))^2}\,dx\sim
\frac{1}{\tilde\ell(y)}$.
\end{prop}

\begin{proof} 
Note that $-(\tilde\ell(x))^{-1}$ is an antiderivative of 
$\ell(x)x^{-1}(\tilde\ell(x))^{-2}$.
\end{proof}

\subsection{Identification of Fourier coefficients}

Write $z=e^{-u+i\theta}$, $u\in[0,1]$, $\theta\in[0,\pi]$.  
Given a function $g_\theta(x)\ge 0$ with $|g_\theta|_\infty\le C\theta^{1-\epsilon}$
for constants $C>0$, $\epsilon\in(0,1)$, such that
$x\to g_\theta(x)(1-G(x))$ is decreasing for each fixed $\theta$, define
\begin{align*}
J_C  =\int_0^\infty e^{-ux}\cos\theta x\, g_\theta(x) (1-G(x))\,dx, 
\quad
J_S  =\int_0^\infty e^{-ux}\sin\theta x\,  g_\theta(x)(1-G(x))\,dx.
\end{align*}
Let $I_C$ and $I_S$ be the corresponding integrals in the case 
$g_\theta\equiv1$.

\begin{prop} \label{prop-IJ}
As $u,\theta\to0^+$,
\begin{alignat*}{2}
|I_S| & \ll \theta u^{-1}\ell(1/u), &
\quad & I_C=\tilde\ell(1/u)(1+o(1))+O(\theta u^{-1}\ell(1/u)), \\
|I_S| & \ll \ell(1/\theta), &  \quad &
I_C=\tilde\ell(1/\theta)(1+o(1))+O(u\theta^{-1}\ell(1/\theta)), \\
|J_S| & \ll \theta^{2-\epsilon} u^{-1}\ell(1/u),
& \quad &  |J_C| \ll 
\theta^{2-\epsilon} u^{-1}\ell(1/u)+\theta^{1-\epsilon}\tilde\ell(1/u), \\
|J_S| & \ll \theta^{1-2\epsilon}, &  \quad &
|J_C| \ll \theta^{1-2\epsilon}+u\theta^{-2\epsilon}.
\end{alignat*}
\end{prop}

\begin{proof}
First,
\begin{align*}
|J_S| & \ll \theta |g_\theta|_\infty \int_0^\infty e^{-ux}\Bigl|\frac{\sin\theta x}{\theta x}\Bigr|\ell(x)\,dx \ll 
\theta |g_\theta|_\infty \int_0^\infty e^{-ux} \ell(x)\,dx \\
& = \theta |g_\theta|_\infty u^{-1}\ell(1/u)\int_0^\infty e^{-\sigma} 
\frac{\ell(\sigma/u)}{\ell(1/u)}\,d\sigma  \ll 
\theta |g_\theta|_\infty u^{-1}\ell(1/u).
\end{align*}
This gives the first estimate for $J_S$ and taking $g_\theta=1$ we obtain the
first estimate for $I_S$.
Alternatively, we make the substitution $\sigma=\theta x$.  Using the oscillation of $\sin\sigma$ and the fact that
$\sigma\mapsto e^{-u\sigma/\theta}g_\theta(\sigma/\theta)(1-G(\sigma/\theta))$
is decreasing,
\begin{align*}
0  \le J_S & =\theta^{-1}\int_0^\infty e^{-u\sigma/\theta} \sin \sigma\, g_\theta(\sigma/\theta)
(1-G(\sigma/\theta))\,d\sigma \\ & \le
\theta^{-1} \int_0^\pi e^{-u\sigma/\theta}\sin \sigma\,  
g_\theta(\sigma/\theta) (1-G(\sigma/\theta))\,d\sigma,
\end{align*}
and so
$|J_S|\le |g_\theta|_\infty\int_0^\pi \sin\sigma\, \ell(\sigma/\theta)\sigma^{-1}\,d\sigma \ll |g_\theta|_\infty \ell(1/\theta)$,
yielding the remaining estimates for $I_S$ and $J_S$.

In the estimates for $J_C$ and $I_C$, we use the fact that 
$\int_0^x (1-G(x))\,dx=\tilde\ell(x)(1+o(1))$.
Note that
\[
\Bigl|\int_{1/u}^\infty e^{-ux}\cos\theta x\, g_\theta(x) (1-G(x))\,dx\Bigr|
\le  |g_\theta|_\infty \int_1^\infty e^{-\sigma} \ell(\sigma/u)\sigma^{-1}\,d\sigma
\ll |g_\theta|_\infty\ell(1/u).
\]
For the integral over $[0,1/u]$, write 
\mbox{$e^{-ux}\cos\theta x = \{e^{-ux}(\cos\theta x -1)\}+\{e^{-ux}-1\} +1$}.
This yields three integrals, the first of which
is estimated by $|g_\theta|_\infty \theta u^{-1}\ell(1/u)$
(like $J_S)$ and the second by $|g_\theta|_\infty \ell(1/u)$.
This leaves $\int_0^{1/u}g_\theta(x)(1-G(x))\,dx\ll |g_\theta|_\infty \int_0^{1/u}(1-G(x))\,dx= |g_\theta|_\infty \tilde\ell(1/u)(1+o(1))$
completing the first estimate for $J_C$.  Setting \mbox{$g_\theta=1$} yields
the first asymptotic expression for $I_C$.
The remaining estimate for $J_C$ is obtained by splitting the range of
integration into $[0,1/\theta]$ and $[1/\theta,\infty)$ and combining the 
above arguments for $J_C$ (first estimate) and $J_S$ (second estimate).
Again the final expression for $I_C$ follows by setting $g_\theta=1$.
\end{proof}

 \begin{cor} \label{cor-IJ}
 Let $z=e^{-u+i\theta}\in B_\epsilon(1)$, $\epsilon$ sufficiently small, $u>0$,
$\theta\ge0$.  Then
\begin{align*}
|1-\lambda(e^{-u+i\theta})|^{-1} & \ll \frac{1}{u\tilde\ell(1/u)},\enspace\text{for
$\theta\in[0,u]$},
\\ 
 |1-\lambda(e^{-u+i\theta})|^{-1} &
\ll \frac{1}{(u+\theta)\tilde\ell(1/\theta)},
\enspace\text{for $\theta\ge u$}.
 \\ |\Re\{(1-\lambda(e^{-u+i\theta}))^{-1}\}| &
\ll \frac{u}{(u^2+\theta^2)\tilde\ell(1/\theta)}+
\frac{\theta\ell(1/\theta)}{(u^2+\theta^2)\tilde\ell(1/\theta)^2},
\enspace\text{for $\theta\ge u$}.
\end{align*}
 \end{cor}

\begin{proof}
Recall from the proof of Lemma~\ref{lem-zestimate} (in particular~\eqref{eq-five}) that
$1-\lambda(z)$ is the sum of five integrals of the form
\[
K=(u-i\theta)\int_0^\infty e^{(-u+i\theta)x}g_z(x)(1-G(x))dx,
\]
where $g_z\ge0$ and either (i) $g_z\equiv1$ or by Corollary~\ref{cor-H1} (ii) $|g_z|_\infty=O(m(1/\theta)\theta)$ as $\theta\to0$.
In case (i), $K=uI_C+\theta I_S+i(uI_S-\theta I_C)$.
In case (ii), $K$ consists of terms of the form $uJ_C$, $uJ_S$, $\theta J_C$, $\theta J_S$.
Hence
\begin{align*}
\Re(1-\lambda(e^{-u+i\theta})) & =uI_C+\theta I_S+E_1, \qquad
\Im\lambda(e^{-u+i\theta})  =uI_S-\theta I_C+E_2, 
\end{align*}
where $E_j=O((|u|+|\theta|)(|J_C|+|J_S|))$, $j=1,2$.

For $\theta\in[0,u]$, 
we use the first estimates in Proposition~\ref{prop-IJ} to obtain
$\Re(1-\lambda(e^{-u+i\theta}))  \sim u\tilde\ell(1/u)$.
Hence $|1-\lambda(e^{-u+i\theta})|\ge |\Re(1-\lambda(e^{-u+i\theta}))|\gg u\tilde\ell(1/u)$.

For $\theta\ge u$ 
we use the second estimates in Proposition~\ref{prop-IJ} to obtain
\begin{align*}
\Re(1-\lambda(e^{-u+i\theta})) & \sim u\tilde\ell(1/\theta)+
O(\theta\ell(1/\theta)),\quad
\Im \lambda(e^{-u+i\theta})   \sim -\theta\tilde\ell(1/\theta).
\end{align*}
Hence $|1-\lambda(e^{-u+i\theta})|\gg (u+\theta)\tilde\ell(1/\theta)$.
Finally, 
\[
 \Re\{(1-\lambda(e^{-u+i\theta}))^{-1}\} =
\frac{\Re\{1-\lambda(e^{-u+i\theta})\}}{|1-\lambda(e^{-u+i\theta})|^2}
\ll \frac{u\tilde\ell(1/\theta)+\theta\ell(1/\theta)}{(u^2+\theta^2)\tilde\ell(1/\theta)^2},
\]
completing the proof.
\end{proof}

\begin{lemma}   \label{lem-R1}
For $u\in[0,1]$, $\theta\in[0,\pi]$, we have 
$\Re T(e^{-u+i\theta}) \ll  h_u(\theta)+g(\theta)$ where 
\[
h_u(\theta)=
\frac{u\tilde\ell(1/u)}{u+\theta}+
\frac{1}{u\tilde\ell(1/u)}1_{[0,u]}(\theta)+
\frac{1}{\tilde\ell(1/\theta)}\frac{u}{u^2+\theta^2},
\quad g(\theta)=
\frac{\ell(1/\theta)}{\tilde\ell(1/\theta)^2}\frac{1}{\theta}.
\]
\end{lemma}

\begin{proof}
By Proposition~\ref{prop-PQ}, for $z=e^{-u+i\theta}\in B_\epsilon(1)$, $\epsilon$ sufficiently small,
\begin{align*}
T(z) & = (1-\lambda(z))^{-1}P+(1-\lambda(z))^{-1}(P(z)-P)+O(1).   
\end{align*}
By Corollary~\ref{cor-IJ}, 
\begin{align*}
\Re\{(1-\lambda(z))^{-1}\} & 
\ll \frac{1}{u\tilde\ell(1/u)}1_{[0,u]}+
\Bigl\{\frac{u}{(u^2+\theta^2)\tilde\ell(1/\theta)}+
\frac{\theta\ell(1/\theta)}{(u^2+\theta^2)\tilde\ell(1/\theta)^2}\Bigr\} 1_{[u,\epsilon]}
 \\ &
 \le \frac{1}{u\tilde\ell(1/u)}1_{[0,u]}+
 \frac{u}{(u^2+\theta^2)\tilde\ell(1/\theta)}+
 \frac{\ell(1/\theta)}{\theta\tilde\ell(1/\theta)^2}.
\end{align*}

By Corollary~\ref{cor-H1},
$P(e^{-u+i\theta})-P(e^{-u}) \ll \tilde\ell(1/\theta)\theta$ uniformly in $u$,
and $P(e^{-u})-P(1)\ll \tilde\ell(1/u)u$.   
Combining this with the estimates for $(1-\lambda(e^{-u+i\theta}))^{-1}$, 
\begin{align*}
(1-\lambda(z))^{-1}(P(z)-P(1))
& \ll \Bigl(\frac{1}{u\tilde\ell(1/u)}1_{[0,u]}+\frac{1}{(u+\theta)\tilde\ell(1/\theta)}1_{[u,\epsilon]}\Bigr)(\theta\tilde\ell(1/\theta)+u\tilde\ell(1/u)) 
\\ & \ll 1+\frac{1}{u\tilde\ell(1/u)}1_{[0,u]}+\frac{u\tilde\ell(1/u)}{u+\theta}.
\end{align*}
This proves the result.
\end{proof}

\begin{rmk} \label{rmk-u=0}    By similar but much simpler calculations, we obtain 
the estimates $|\Re\{(1-\lambda(e^{i\theta}))^{-1}\}|\ll g(\theta)$ for $\theta\in(0,\epsilon)$ and
$\Re T(e^{i\theta})\ll g(\theta)$ for $\theta\in(0,\pi]$.
\end{rmk}

\begin{cor} \label{cor-R1}  
For $n\ge1$,
$\lim_{u\to 0}\int_0^\pi \cos n\theta\Re T(e^{-u+i\theta})\,d\theta =
\int_0^\pi \cos n\theta \Re T(e^{i\theta})\,d\theta$.
 \end{cor}

\begin{proof}
The function $g(\theta)=\frac{\ell(1/\theta)}{\tilde\ell(1/\theta)^2}\frac{1}{\theta}$ lies in $L^1$ by Proposition~\ref{prop-elltilde}.
Note that $\Re T(e^{-u+i\theta})\to\Re T(e^{i\theta})$ 
and $h_u(\theta)\to 0$ pointwise (for all $\theta\neq0$).
We claim that $\int_0^\pi h_u(\theta)\,d\theta\to0$.
The result then follows from the dominated convergence theorem
(more precisely the extended version stated 
in~\cite[p.~92]{Royden}).

The claim is easy to check for the first two terms in $h_u$.
For the third term $k_u(\theta)
=\frac{1}{\tilde\ell(1/\theta)}\frac{u}{u^2+\theta^2}$, we
compute for $b\in(0,\pi)$ that
\begin{align*}
\int_0^\pi k_u(\theta)\,d\theta & = \int_0^b k_u(\theta)\,d\theta
+ \int_b^\pi k_u(\theta)\,d\theta
\ll \frac{1}{\tilde\ell(1/b)}u^{-1}b+ u\int_b^\pi \theta^{-2}\,d\theta
\\ & \ll \frac{1}{\tilde\ell(1/b)}u^{-1}b + ub^{-1},
\end{align*}
where the implied constant is independent of $b$ and $u$.
Define $b=b(u)$ such that $u=b(\tilde\ell (1/b))^{-1/2}$.
In particular, $b\to0$ as $u\to0$ and 
so $\int_0^\pi k_u(\theta)\,d\theta\ll (\tilde\ell(1/b))^{-1/2}\to0$ as required.
\end{proof}

\begin{cor}   \label{cor-Fourier_beta=1}
$\Re T\in L^1$ and
 $T_n=\frac{2}{\pi}\int_0^\pi \cos n\theta\,\Re T(e^{i\theta})\,d\theta$
for all $n\ge1$.
\end{cor}

\begin{proof}
The function $\theta\mapsto T(\rho e^{i\theta})$ is integrable for each fixed
$\rho<1$.   Moreover, the power series for $T(z)$ is uniformly 
convergent on compact subsets of $\D$, so we obtain
\[
\SMALL \int_0^\pi \cos n\theta\Re T(\rho e^{i\theta})\,d\theta
=\sum_{j=0}^\infty T_j\rho^j\int_0^\pi \cos n\theta
\cos j\theta\,d\theta = \frac{\pi}{2} T_n\rho^n.
\]
By Corollary~\ref{cor-R1},
$T_n= \frac{2}{\pi}\rho^{-n}\int_0^\pi \cos n\theta\Re T(\rho e^{i\theta})\,d\theta\to 
 \frac{2}{\pi}\int_0^\pi \cos n\theta\Re T(e^{i\theta})\,d\theta$,
as $\rho=e^{-u}\to1$.
\end{proof}

\subsection{Asymptotics of $T_n$}

The calculations in this subsection are restricted to the unit circle,
so we revert to writing $T(\theta)$ instead of $T(e^{i\theta})$ and so on.
First we determine the asymptotics of $\Re\{(1-\lambda(\theta))^{-1}\}$
(see also~\cite{AaronsonDenker01}).

\begin{lemma} \label{lem-one}
$\Re \{(1-\lambda(\theta))^{-1}\}= \frac{\pi}{2}g(\theta)(1+o(1))$
as $\theta\to0^+$,
where $\BIG g(\theta)=\frac{\ell(1/\theta)}{\theta(\tilde\ell(1/\theta))^2}$.
\end{lemma}

\begin{proof}
By Remark~\ref{rmk-u=0},
$\Re \{(1-\lambda(\theta))^{-1}\}\ll g(\theta)$.
We claim that $I_S\sim \frac{\pi}{2}\ell(1/\theta)$ from which the result
follows easily.

Let $M\ge 3\pi$.
Since $\sigma\mapsto \ell(\sigma)/\sigma$ is decreasing, 
we have the oscillatory integral estimate
\[
\frac{1}{\ell(1/\theta)}I_S=\int_0^\infty \frac{\sin \sigma}{\sigma}
\frac{\ell(\sigma/\theta)}{\ell(1/\theta)}\,d\sigma
=\int_0^M \frac{\sin \sigma}{\sigma}
\frac{\ell(\sigma/\theta)}{\ell(1/\theta)}\,d\sigma+F_M,
\]
where
\[
|F_M|\le 2\sup_{\sigma\in[M-2\pi,M+2\pi]}\frac{\ell(\sigma/\theta)}{\sigma\ell(1/\theta)}.
\]
By Potter's bounds, for any $\delta>0$, $F_M=O(1/M^{1-\delta})$.
Hence, $\lim_{\theta\to0}\ell(1/\theta)^{-1}I_S=\int_0^M \frac{\sin \sigma}{\sigma} \,d\sigma+O(1/M^{1-\delta})$.   Let $M\to\infty$ to verify the claim.
\end{proof}

\begin{cor} \label{cor-one}
Let $a>0$.   Then
$\lim_{n\to\infty} \tilde\ell(n)\int_0^{a/n}\Re\{(1-\lambda(\theta))^{-1}\}\,
d\theta =\frac{\pi}{2}$.
\end{cor}

\begin{proof}   By Lemma~\ref{lem-one}, we can write 
$\Re\{(1-\lambda(\theta))^{-1}\}=\frac{\pi}{2}g(\theta)(1+h(\theta))$
where $h(\theta)=o(1)$ as $\theta\to0^+$.   
Let $H(n)=\sup_{\theta\in[0,a/n]}|h(\theta)|$, so
$H(n)=o(1)$ as $n\to\infty$.
Then
\[
\int_0^{a/n}\Re\{(1-\lambda(\theta))^{-1}\}\,d\theta
=\frac{\pi}{2}\int_0^{a/n}g(\theta)\,d\theta
+O\Bigl(H(n)\int_0^{a/n}g(\theta)\,d\theta\Bigr).
\]
By Proposition~\ref{prop-elltilde}, 
$\int_0^{a/n}g(\theta)\,d\theta=\tilde\ell(n/a)^{-1}$.   Hence
\[
\tilde\ell(n)\int_0^{a/n}\Re\{(1-\lambda(\theta))^{-1}\}\,d\theta
=\frac{\pi}{2}\frac{\tilde\ell(n)}{\tilde\ell(n/a)}(1+o(1))\to\frac{\pi}{2},
\]
as $n\to\infty$.
\end{proof}

\begin{pfof}{Theorem~\ref{thm-conv}, $\beta=1$.}
By Remark~\ref{rmk-u=0} and Proposition~\ref{prop-Karamata}(b),
$T(\theta) \ll (\theta\tilde\ell(1/\theta))^{-1}$.
Let $\delta>0$.  By the argument in the proof of Lemma~\ref{lem-GL}, we obtain
\begin{align} \label{eq-est1}
\SMALL \tilde\ell(n)\int_{a/n}^{\pi} \cos n\theta\, T(\theta)\,d\theta
\ll a^{-(1-\delta)},
\end{align}
for $a\in[1,n]$, $n\ge1$.  Also, we have
\begin{align} \label{eq-est2}
\SMALL
\tilde\ell(n)\int_0^{a/n} \cos n\theta\,T(\theta)\,d\theta=
\tilde\ell(n)\int_0^{a/n} \cos n\theta\, (1-\lambda(\theta))^{-1}\,d\theta\, P + 
O(an^{-(1-\delta)}).
\end{align}
By Lemma~\ref{lem-one},
\begin{align} \label{eq-est3} \nonumber
\SMALL &\SMALL  \tilde\ell(n)\int_0^{a/n} (\cos n\theta-1)\Re\{(1-\lambda(\theta))^{-1}\} 
\,d\theta
\\[.75ex] &\SMALL  \qquad \qquad  \ll 
\tilde\ell(n)\int_0^a (\cos \sigma-1)\ell(n/\sigma)(\tilde\ell(n/\sigma))^{-2}\sigma^{-1}\,d\sigma 
\ll \ell(n)(\tilde\ell(n))^{-1}a.
\end{align}

Combining estimates~\eqref{eq-est1},~\eqref{eq-est2}
and~\eqref{eq-est3} with Corollary~\ref{cor-one}, we obtain
\[
\lim_{a\to\infty}\lim_{n\to\infty} \tilde\ell(n)\int_0^\pi \cos n\theta\,
\Re T(\theta)\,d\theta = \frac{\pi}{2}P,
\]
and hence $\tilde\ell(n)T_n\to P$ by Corollary~\ref{cor-Fourier_beta=1}.
\end{pfof}

\section{Pointwise dual ergodicity}
\label{sec-pde}

In this section, we give an elementary proof of pointwise dual ergodicity
for the class of systems under consideration for all $\beta\in(0,1]$.
We assume our general framework from Section~\ref{sec-ORT}, except that
we do not require H2(ii).

\begin{prop} \label{prop-realAD}
$T(s)\sim \begin{cases}
\tilde\ell(\frac{1}{1-s})^{-1}(1-s)^{-1}P, & \beta=1, \\
 \Gamma(1-\beta)^{-1}\ell(\frac{1}{1-s})^{-1}(1-s)^{-\beta}P, & \beta\in(0,1),
\end{cases}$
as $s\to 1^-$.
\end{prop}

\begin{proof}
This is similar to the proof of Lemma~\ref{lem-AD}, but much
simpler since the integrals are absolutely convergent.
By Proposition~\ref{prop-PQ}, for $s\in(1-\epsilon,1]$,
\[
T(s)=(1-\lambda(s))^{-1}P+
(1-\lambda(s))^{-1}(P(s)-P)+O(1).
\]
By Corollary~\ref{cor-H1}, $P(s)-P\ll m(\frac{1}{1-s})(1-s)^\beta$,
so it suffices to establish the desired asymptotic expression for 
$(1-\lambda(s))^{-1}$.

Setting $s=e^{-u}$ we have
$\lambda(s)=1+\int_0^\infty (e^{-ux}-1)\hat v_u(x)dG(x)$,
where $G(x)=\mu(\varphi\le x)$ and $|\hat v_u-1|_\infty = o(1)$ as $u\to0$.
Writing $d\hat G_u=v_u\,dG$ and integrating by parts, 
\begin{align} \label{eq-real_lambda}
\lambda(s)=1+\int_0^\infty (e^{-ux}-1)d\hat G_u(x)
=1-u\int_0^\infty e^{-ux}g_u(x)(1-G(x))\,dx,
\end{align}
where $|g_u(x)-1|_\infty = o(1)$ as $u\to0$.

If $\beta\in(0,1)$, then
\[
\lambda(s)=1-\ell(1/u) u^\beta
\int_0^\infty e^{-\sigma} g_u(\sigma/u)\{\ell(\sigma/u)\ell(1/u)^{-1}\}
\sigma^{-\beta}\,du.
\]
By the dominated convergence theorem,
$\lambda(s)=1-\Gamma(1-\beta) \ell(1/u) u^\beta(1+o(1))$.
The result follows since $u=-\log s = 1-s+O((1-s)^2)$.

If $\beta=1$, then picking up from~\eqref{eq-real_lambda}, 
\begin{align*}
&  \SMALL \int_0^\infty e^{-ux}g_u(x)(1-G(x))\,dx
=\int_0^{1/u} g_u(x)(1-G(x))\,dx  \\ & \SMALL \qquad
+\int_0^{1/u} (e^{-ux}-1)g_u(x)(1-G(x))\,dx
+\int_{1/u}^\infty e^{-ux}g_u(x)(1-G(x))\,dx.
\end{align*}
The last two integrals are $O(\ell(1/u))$, and
$\int_0^{1/u} g_u(x)(1-G(x))\,dx=(1+o(1))\int_0^{1/u}(1-G(x))\,dx\sim 
\tilde\ell(1/u)$ by definition of $\tilde\ell$.
\end{proof}

\begin{thm} \label{thm-pde}  
Let $v\in L^1(X)$ and $\beta\in(0,1]$.   Then
\[
\lim_{n\to\infty}m(n) n^{-\beta}\sum_{j=1}^n L^jv = \beta^{-1}d_\beta 
\int_X v\,d\mu,\enspace\text{almost everywhere on $X$}.
\]
\end{thm}

\begin{proof}
By~\cite[Proposition~3.7.6]{Aaronson} (see also the proof of
Proposition~\ref{prop-pi}), it suffices to prove pointwise
dual ergodicity on the tower $\Delta$ defined in Section~\ref{sec-tower}.    
Let $L_\Delta:L^1(\Delta)\to L^1(\Delta)$ denote the transfer operator on 
$\Delta$.   Note that $T_nv=1_YL_\Delta^n(1_Yv)$ coincides with our usual $T_n$.
By the Hurewicz ergodic theorem~\cite{Hurewicz44}, it is enough to prove 
that $m(n)n^{-\beta}\sum_{j=0}^{n-1}L_\Delta^jv \to 
\beta^{-1}d_\beta\int_\Delta v\,d\mu$ almost everywhere on
$\Delta$ for the particular choice $v=1_Y$.     

Let $y\in Y$.  For $\beta\in(0,1)$, Proposition~\ref{prop-realAD} gives
 $(T(s)v)(y)\sim \Gamma(1-\beta)
 \ell(\frac{1}{1-s})^{-1}(1-s)^{-\beta}\int_Y v\,d\mu$ as $s\to 1^-$.
By (the discrete version of) the Karamata Tauberian 
Theorem~\cite{Karamata33}~\cite[p.~445]{Feller66},~\cite[Proposition~4.2]{ThalerZweimuller06}, it follows that 
$\sum_{j=1}^n (T_jv)(y)\sim \beta^{-1} d_\beta\ell(n)^{-1}n^\beta\int_Yv\,d\mu$ 
as $n\to\infty$.   Similarly for $\beta=1$.

Finally, let $p=(y,j)$ be a general point in $\Delta$.
Then $(L_\Delta^nv)(p)=(L_\Delta^{n-j}v)(y,0)=(T_{n-j}v)(y)$ for all $n>j$.
\end{proof}

An immediate consequence is the Darling-Kac law~\cite{DarlingKac57}.
Recall that a random variable $\mathcal{M}_\beta$ on $(0,\infty)$ has the 
{\em normalised Mittag-Leffler distribution of order $\beta$} if  $E(e^{z\mathcal{M}_\beta)}=\sum_{p=0}^\infty \Gamma(1+\beta)^pz^p/\Gamma(1+p\beta)$ 
for all $z\in\C$.

\begin{cor}
\label{cor-pde}
Let $v\in L^1(X)$, $v\ge0$, $\int_X v\,d\mu=1$, and let $\beta\in(0,1]$.  Then
\[
m(n)n^{-\beta}\sum_{j=1}^n v\circ f^j 
\to_d \mathcal{M}_\beta\enspace\text{as $n\to\infty$}.
\]
The convergence is in the sense of {\em strong
distributional convergence}: convergence in distribution under
any probability measure absolutely continuous w.r.t.\ $\mu$.
\end{cor}

\begin{proof}   This follows from Theorem~\ref{thm-pde} by 
Aaronson~\cite{Aaronson81},~\cite[Corollary~3.7.3]{Aaronson}.
\end{proof}
 
To conclude the section, we mention a simple 
consequence of Theorem~\ref{thm-conv} which gives uniform convergence on $Y$ in
the pointwise dual ergodic theorem for $\beta>\frac12$ for sufficiently
regular observables.

\begin{prop} \label{prop-pde}  
If $\beta\in(\frac12,1]$, then
$\lim_{n\to\infty}m(n) n^{-\beta}\sum_{j=1}^n T_j = \beta^{-1}d_\beta P$.
\end{prop}

\begin{proof}
By Theorem~\ref{thm-conv},
$T_n=m(n)^{-1}n^{-(1-\beta)}d_\beta P+S_n$
where $\|S_n\|=o(m(n)^{-1}n^{-(1-\beta)})$.  Hence
\begin{align} \label{eq-pde}
m(n)n ^{-\beta}\sum_{j=1}^n T_j = 
m(n)n ^{-\beta}\sum_{j=1}^n m(j)^{-1}j^{-(1-\beta)}\,d_\beta P +
m(n)n ^{-\beta}\sum_{j=1}^n S_j.
\end{align}
By Proposition~\ref{prop-Karamata}(a), 
$\sum_{j=1}^n m(j)^{-1}j^{-(1-\beta)}\sim \beta^{-1}m(n)^{-1}n^\beta$, 
so the first term on the RHS of~\eqref{eq-pde} converges to the desired limit
$\beta^{-1}d_\beta P$.

Let $\delta>0$, and choose $n_0$ such that
$\|S_n\|\le \delta m(n)^{-1}n^{-(1-\beta)}$ for $n>n_0$.
Then $\sum_{j=1}^n \|S_j\| \le \sum_{j=1}^{n_0}\|S_j\|+\sum_{j=n_0+1}^n 
\delta m(j)^{-1}j^{-(1-\beta)}$.   Applying Proposition~\ref{prop-Karamata}(a) 
once more, we obtain $\limsup_{n\to\infty}m(n)n^{-\beta}\sum_{j=1}^n \|S_j\|\le
\beta^{-1}\delta$. Since $\delta>0$ is arbitrary, 
the second term on the RHS of~\eqref{eq-pde} converges to zero.   
\end{proof}

\section{Results for $\beta\in(0,\frac12]$}
\label{sec-zerohalf}

In this section, we prove Theorems~\ref{thm-bound} and~\ref{thm-liminf}.

\begin{pfof}{Theorem~\ref{thm-bound}}
If $\beta\in(0,\frac12]$, then the proof of Theorem~\ref{thm-conv}
breaks down only in the estimation of $I_3$ in Lemma~\ref{lem-GL}.

\noindent(a) When $\beta=\frac12$, 
it is evident from the proof of Lemma~\ref{lem-GL} that
$I_3\ll \ell(n)n^{-\frac12} 
\int_{1/n}^\pi \ell(1/\theta)^{-2}\theta^{-1}\,d\theta$.
The remaining estimates are $O(\ell(n)^{-1}n^{-\frac12})$ as before.
By Proposition~\ref{prop-Karamata}(b), 
$\ell(n)^2\int_{1/n}^\pi \ell(1/\theta)^{-2}\theta^{-1}\,d\theta\to\infty$ 
as $n\to\infty$.  Hence the estimate for $I_3$ is the dominant one.

\noindent (b) For $\beta\in(0,\frac12)$, 
$I_3\ll \ell(n) n^{-\beta}\int_0^\pi \ell(1/\theta)^{-2}\theta^{-2\beta}\,d\theta\ll \ell(n) n^{-\beta}$.   The remaining estimates are $O(\ell(n)^{-1}n^{-(1-\beta)})$ as before.

\noindent (c)  If $Pv=0$, then $\|T(\theta)v\|\ll \|v\|$.  Hence
the resolvent identity
\[
\{T(\theta)-T(\theta-\pi/n)\}v=T(\theta)(R(\theta)-R(\theta-\pi/n))T(\theta-\pi/n)v
\]
yields $\|\{T(\theta)-T(\theta-\pi/n)\}v\|\ll \ell(1/\theta)^{-1}\theta^{-\beta}\ell(n)n^{-\beta}$.
It follows that
\[
2T_nv = \int_0^{2\pi}\{T(\theta)-T(\theta-\pi/n)\}e^{-in\theta}\,d\theta
\ll \ell(n)n^{-\beta}\|v\|,
\]
as required.
\end{pfof}

Next, we establish the lower bound in Theorem~\ref{thm-liminf}(b).

\begin{prop}   \label{prop-liminf}
If $\beta\in(0,1)$, then 
$\liminf_{n\to\infty} \ell(n)n^{1-\beta}T_nv \ge  d_\beta \int_Y v\,d\mu$ pointwise on $Y$
for all $v\ge0$.
\end{prop}

\begin{proof}
For any $m\ge1$, we can write
\[
T=(I-R)^{-1}=I+R+\cdots+R^{m-1}+T^{(m)}, \quad T^{(m)}=R^m(I-R)^{-1}.
\]
Since $R$ is a positive operator, we deduce that
$(T_nv)(y) \ge (T^{(m)}_nv)(y)$ for all $v\ge0$, $y\in Y$, $m\ge1$.
Choosing $m=b_n\sim b\ell(n)^{-1}n^\beta$, $b>0$,  as 
in~\cite[Theorem~3.6.1]{GarsiaLamperti62}, we obtain
$\ell(n) n^{1-\beta}T^{(b_n)}_n \sim d_bP$, where $d_b\to d_\beta$ as $b\to 0$,
and the result follows.
\end{proof}

The following result is well-known (see~\cite[Theorem~2.9.1]{BinghamGoldieTeugels},~\cite{GarsiaLamperti62})
but stated in a slightly different form, so we provide the proof for completeness.

\begin{prop}  \label{prop-zerodensity}
Let $f_n$ be a sequence in $\R$
and let $A\in\R$, Suppose that $\beta\in(0,1)$, that $\ell(n)$ is slowly
varying, and that
\begin{itemize}
\item[(a)] $\liminf_{n\to\infty}\ell(n)n^{1-\beta}f_n\ge A$,
\item[(b)] $\lim_{n\to\infty}\ell(n)n^{-\beta}\sum_{j=1}^n f_j=\beta^{-1}A$.
\end{itemize}
Then there exists a set $E$ of density zero such that
$\lim_{n\to\infty,\;n\not\in E}\ell(n)n^{1-\beta}f_n=A$.

In particular, 
$\liminf_{n\to\infty}\ell(n)n^{1-\beta}f_n=A$.
\end{prop}

\begin{proof}
Our proof is modelled on \cite[p.~65, Lemma~6.2]{Petersen}.

By Proposition~\ref{prop-Karamata}(a), 
$\sum_{j=1}^n\ell(j)^{-1}j^{-(1-\beta)}\sim \beta^{-1}\ell(n)^{-1} n^\beta$.
Let $\hat f_n=f_n-\ell(n)^{-1}n^{(\beta-1)}A$.  
Then (b) is equivalent to $\lim_{n\to\infty}\ell(n)n^{-\beta}\sum_{j=1}^n 
\hat f_j=0$.
Hence we may suppose without loss that $A=0$.
In addition, there is a monotone increasing function $g(n)$ such that
$\ell(n)n^{1-\beta}\sim g(n)$ (see for 
example~\cite[Theorem 1.5.3]{BinghamGoldieTeugels}).
Hence we may suppose that $\ell(n)n^{1-\beta}$ is increasing.

Define the nested sequence of sets 
$E_q=\{n\ge1:\ell(n)n^{1-\beta}f_n>1/q\}$.
We claim that each $E_q$ has density zero.  Let $\delta>0$.  By (a),
there exists $n_0\ge1$ such that $\ell(n)n^{1-\beta}f_n>-\delta$ for all
$n\ge n_0$.    Hence
\begin{align*}
\frac1n \sum_{j=1}^n 1_{E_q}(j) & \le
\ell(n) n^{-\beta}\sum_{j=1}^n \{\ell(j)j^{(1-\beta)}\}^{-1}1_{E_q}(j)
\le q\ell(n)n^{-\beta}\sum_{j=1}^n f_j1_{E_q}(j) \\
& = q\ell(n)n^{-\beta}\Bigl(\sum_{j=1}^n f_j 
- \sum_{j=1}^{n_0} f_j1_{E_q^c}(j)
- \sum_{n_0+1}^n f_j1_{E_q^c}(j)\Bigr) \\
& \le  q\ell(n)n^{-\beta}\sum_{j=1}^n f_j 
+ q\ell(n)n^{-\beta} \sum_{j=1}^{n_0} |f_j| 
+ q\ell(n)n^{-\beta}\sum_{n_0+1}^n \ell(j)^{-1}j^{\beta-1}\delta \\ 
& = 
q\ell(n)n^{-\beta}\sum_{j=1}^n f_j +O(\ell(n)n^{-\beta})+O(\delta).
\end{align*}
By (b), $\limsup_{n\to\infty}
\frac1n \sum_{j=1}^n 1_{E_q}(j)=O(\delta)$ and the claim follows since
$\delta$ is arbitrary.

By the claim, there exist $1=i_0<i_1<i_2<\cdots$ such that
$\frac1n\sum_{j=1}^n 1_{E_q}(j)<1/q$ for $n\ge i_{q-1}$, $q\ge2$.  Let 
$E=\bigcup_{q=1}^\infty E_q\cap(i_{q-1},i_q)$.
If $n\in E$ and $n\le i_q$, then $n\in E_q$.
Hence for $i_{q-1}\le n\le i_q$, we have
$\frac1n\sum_{j=1}^n 1_E(j)  \le \frac1n\sum_{j=1}^n 1_{E_q}(j)  \le 1/q$,
verifying that $E$ has density zero.

On the other hand, 
if $n\not\in E$ and $i_{q-1}<n<i_q$, then $n\notin E_q$ and so
$\ell(n)n^{1-\beta}f_n\le 1/q$.
Hence $\limsup_{n\to\infty,\;n\not\in E} \ell(n)n^{1-\beta}f_n\le 0$.
Combined with assumption~(a), we deduce that 
$\lim_{n\to\infty,\;n\not\in E} \ell(n)n^{1-\beta}f_n=0$ and the last statement follows
immediately.~
\end{proof}

\begin{pfof}{Theorem~\ref{thm-liminf}}
Part (b) is stated for $v\ge0$ and
in part (a) we can break $v$ into positive and negative parts.  Hence
without loss we may suppose that $v\ge0$.

By Proposition~\ref{prop-liminf} and Theorem~\ref{thm-pde},
we have verified the hypotheses of Proposition~\ref{prop-zerodensity}
with $f_n=(T_nv)(y)$ and $A=d_\beta\int_Yv\,d\mu$.
The result is immediate.
\end{pfof}

\section{Second order asymptotics}
\label{sec-second}

In this section we prove results on second order asymptotics and higher 
order asymptotic expansions under assumptions on the 
asymptotics of $\mu(\varphi>n)$.   Throughout, we suppose that $\ell(n)$
is asymptotically constant (and that $\beta>\frac12$).
In Subsection~\ref{sec-second_half}, we consider the case when 
$\beta\in(\frac12,1)$.   The case $\beta=1$ is covered in
Subsection~\ref{sec-second_one}.    
Error terms in the Dynkin-Lamperti arcsine law are obtained in 
Subsection~\ref{sec-arcsine}.

\subsection{Second order asymptotics for $\beta\in(\frac12,1)$}
\label{sec-second_half}

We assume that $\mu(\varphi>n)= c(n^{-\beta}+H(n))$, where 
$H(n)=O(n^{-2\beta})$ and $c>0$.    
(It is easy to relax this to the more general hypothesis
that $H(n)=O(n^{-q})$, $q>1$.   However the formulas
become more complicated and our assumption is satisfied by~\eqref{eq-LSV}.)

Recall that $c_H=\int_0^\infty H_1(x)\,dx$ where
$H_1(x)=[x]^{-\beta}-x^{-\beta}+H([x])$.
Define $\xi^\pm_p=\int_0^\infty e^{\pm i\sigma}\sigma^{-p}\,d\sigma$,
$0<p<1$, so $c_\beta=-i\xi^+_\beta$, and recall that
$e_0=ic_H/c_\beta$.

Set $d_{\beta,j}'=e_0^j\xi^-_{(j+1)\beta-j}/c_\beta=
ie_0^j\xi^-_{(j+1)\beta-j}/\xi^+_\beta$ and
$d_{\beta,j}=\frac{1}{\pi}\Re d_{\beta,j}'$.

We note that $d_{\beta,0}=d_\beta=\frac{1}{\pi}\sin\beta\pi>0$, and that
either $d_{\beta,j}=0$ for all $j\ge1$ or $d_{\beta,j}\neq0$ for all $j\ge1$.  
Moreover, the latter situation is typical.

\begin{thm} \label{thm-second}   
Suppose that $\beta\in(\frac12,1)$ and that
$\mu(\varphi>n)= c(n^{-\beta}+H(n))$, where $H(n)=O(n^{-2\beta})$ and $c>0$.   
Let $\gamma=\min\{1-\beta,\beta-\frac12\}$.  Then 
\[
n^{1-\beta}T_n=c^{-1}d_\beta P + O(n^{-\gamma}).
\]

Moreover, if $\beta\in(\frac34,1)$, then
$\lim_{n\to\infty}n^{1-\beta}\{n^{1-\beta}T_n-c^{-1}d_\beta P \}=d_{\beta,1}P$.
\end{thm}

\begin{rmk}
For $\beta$ close to $1$, we obtain higher order asymptotic expansions.
There exist constants $d_{\beta,j}\in\R$, $j\ge0$ with $d_{\beta,0}=d_\beta$,
such that for each $q=0,1,2,\ldots$, 
\[
n^{1-\beta}T_n=c^{-1}\Bigl\{\sum_{j=0}^q d_{\beta,j} n^{-j(1-\beta)} + 
O(n^{-(q+1)(1-\beta)})\Bigr\}P+ O(n^{-(\beta-\frac12)}).
\]
Thus, 
$n^{1-\beta}T_n = \begin{cases} c^{-1}d_\beta P+O(n^{-(\beta-\frac12)}), & \beta\in(\frac12,\frac34] \\[.75ex]
c^{-1}d_\beta P+c^{-1}d_{\beta,1} n^{-(1-\beta)}P+O(n^{-(\beta-\frac12)}), & 
\beta\in(\frac34,\frac56] \end{cases}\enspace$
and so on.
\end{rmk}

\begin{cor} \label{cor-pde_second}   
$n^{-\beta}\sum_{j=1}^n T_j=c^{-1}\beta^{-1}d_\beta P + O(n^{-\gamma})$
uniformly on $Y$.
\end{cor}

\begin{proof}
Specialising the proof of Proposition~\ref{prop-pde}, we have
\begin{align*} 
n ^{-\beta}\sum_{j=1}^n T_j = 
n ^{-\beta}\sum_{j=1}^n j^{-(1-\beta)}\,c^{-1}d_\beta P +
n ^{-\beta}\sum_{j=1}^n S_j,
\end{align*}
where $S_j=O(j^{-(1-\beta+\gamma)})$.
Now $\sum_{j=1}^n j^{-(1-\beta)}=\int_1^n x^{-(1-\beta)}\,dx+O(1)=
\beta^{-1}n^\beta+O(1)$, and $\sum_{j=1}^n S_j=O(n^{\beta-\gamma})$.
\end{proof}

In the remainder of this subsection, we prove Theorem~\ref{thm-second}.

\begin{prop}    \label{prop-lambda}
$(1-\lambda(\theta))^{-1}= c^{-1}c_\beta^{-1}\sum_j e_0^j\theta^{-((j+1)\beta-j)}+O(1)$,
where the sum is over those $j\ge0$ with $(j+1)\beta-j>0$.
\end{prop}

\begin{proof}
By Lemma~\ref{lem-ADerror}, 
$1-\lambda(\theta)=cc_\beta\theta^\beta(1-e_0\theta^{1-\beta}+O(\theta^{\beta}))$.
Now invert and note that $(1-e_0\theta^{1-\beta}+O(\theta^{\beta}))^{-1}
=\sum_{j\ge0} e_0^j\theta^{j(1-\beta)}+O(\theta^{\beta})$.
\end{proof}

\begin{prop} \label{prop-second}
Let $n\ge1$, $0<a<\epsilon n$.  Then
\begin{align*}
& n^{1-\beta}\int_0^{a/n}(1-\lambda(\theta))^{-1} e^{-in\theta}\,d\theta \\
& \qquad\qquad =
c^{-1}\sum_j d_{\beta,j}'n^{-j(1-\beta)}+O(\sum_jn^{-j(1-\beta)}a^{-((j+1)\beta-j)}) + 
O(an^{-\beta}),
\end{align*}
and the sums are over those $j\ge0$ with $(j+1)\beta-j>0$.
\end{prop}

\begin{proof}
By Proposition~\ref{prop-lambda},
\begin{align*}
\int_0^{a/n}(1-\lambda(\theta))^{-1} e^{-in\theta}\,d\theta 
& = c^{-1}c_\beta^{-1}\sum_j e_0^j\int_0^{a/n}\theta^{-((j+1)\beta-j)}e^{-in\theta}\,d\theta
+O(a/n) \\
& = c^{-1}c_\beta^{-1}\sum_j e_0^jn^{-(j+1)(1-\beta)}\int_0^a\sigma^{-((j+1)\beta-j)}e^{-i\sigma}\,d\sigma
+O(a/n).
\end{align*}
Hence
\begin{align*}
& n^{1-\beta} \int_0^{a/n}(1-\lambda(\theta))^{-1} e^{-in\theta}\,d\theta  \\
& \qquad   = c^{-1}\sum_jd_{\beta,j}'n^{-j(1-\beta)}-
c^{-1}c_\beta^{-1}\sum_je_0^jn^{-j(1-\beta)}\int_a^\infty \sigma^{-((j+1)\beta-j)}e^{-i\sigma}\,d\sigma +O(an^{-\beta}),
\end{align*}
yielding the required result.
\end{proof}

\begin{pfof}{Theorem~\ref{thm-second}}
By Lemma~\ref{lem-GL}, $n^{1-\beta}\int_{a/n}^\pi T(\theta)e^{-in\theta}\,d\theta \ll
a^{-(2\beta-1)}$.  For $\theta\in[0,a/n]\subset [0,\epsilon]$, 
$T(\theta)  =(1-\lambda(\theta))^{-1}P+O(1)$ by Lemma~\ref{lem-AD}(b).
Hence,
\[
n^{1-\beta}\int_0^\pi T(\theta)e^{-in\theta}\,d\theta = 
n^{1-\beta}\int_0^{a/n}(1-\lambda(\theta))^{-1}e^{-in\theta}\,d\theta \,P
+ O(an^{-\beta})+O(a^{-(2\beta-1)}). 
\]
Taking $a=n^{1/2}$ we obtain the error term $O(n^{-(\beta-\frac12)})$ and 
the result follows from Proposition~\ref{prop-second} and 
Corollary~\ref{cor-Fourier}.
\end{pfof}

\subsection{Second order asymptotics for $\beta=1$}
\label{sec-second_one}

\begin{thm} \label{thm-second_one}   
Suppose that $\mu(\varphi>n)= c(n^{-1}+H(n))$ where 
$c>0$ and $H(n)=O(n^{-q})$, $q>1$.  
Let $H_1(x)=[x]^{-1}-x^{-1}+H([x])$, $x\ge1$ and $H_1(x)=\frac{1}{c}$,
$x\in[0,1)$.  Then 
\[
\SMALL (\log n)T_n = c^{-1}\bigl\{1 -\int_0^\infty 
H_1(x)\,dx\,(\log n)^{-1}  + O((\log n)^{-2})\bigr\}P+O((\log n)^\frac12 n^{-\frac12}).
\]
\end{thm}

\begin{cor} \label{cor-pde_second_one}
$(\log n)n^{-1}\sum_{j=1}^n T_j=c^{-1}P + O((\log n)^{-1})$ uniformly on $Y$.
\end{cor}

\begin{proof}
As in the proof of Corollary~\ref{cor-pde_second}, we have
\begin{align*} 
(\log n)n ^{-1}\sum_{j=1}^n T_j = 
(\log n)n ^{-1}\sum_{j=1}^n (\log j)^{-1}c^{-1}P +
(\log n)n ^{-1}\sum_{j=1}^n S_j,
\end{align*}
where $\|S_j\|=O((\log j)^{-2})$.  Integration by parts yields
 $\sum_{j=1}^n (\log j)^{-1}=n(\log n)^{-1}+\int_2^n (\log x)^{-2}\,dx+O(1)=
n(\log n)^{-1}+O(n(\log n)^{-2})$ while $\sum_{j=1}^n S_j\ll n(\log n)^{-2}$.
\end{proof}

In the remainder of this subsection, we prove Theorem~\ref{thm-second_one}.

\begin{prop} \label{prop-second_one}
Let $c_H=\int_0^1(\cos\sigma-1)\sigma^{-1}\,d\sigma
+ \int_1^\infty \cos\sigma\,\sigma^{-1}\,d\sigma+ \int_0^\infty H_1(x)\,dx$.
Then
\[
\SMALL \Re\{(1-\lambda(\theta))^{-1}\}=
c^{-1}\frac{\pi}{2}\theta^{-1} (\log\frac{1}{\theta})^{-2}-
c^{-1}c_H\pi\theta^{-1} (\log\frac{1}{\theta})^{-3}+
O(\theta^{-1}(\log\frac{1}{\theta})^{-4}).
\]
\end{prop}

\begin{proof}
Without loss of generality, we may suppose that $q\in(1,2)$.
Recall that $G(x)\equiv0$ for $x\in[0,1)$ and
$1-G(x)=c(x^{-1}+H_1(x))$ for $x\ge1$ where $H_1(x)=O(x^{-q})$.
In particular, $H_1\in L^1$.  Write
\[
\SMALL I_C  = \int_0^\infty \cos\theta x (1-G(x))\,dx
= c\int_1^\infty \cos\theta x \,x^{-1}\,dx
+c\int_0^\infty \cos\theta x \,H_1(x)\,dx.
\]
Now,
\begin{align*}
\SMALL \int_1^\infty \cos\theta x \,x^{-1}\,dx & \SMALL = \int_1^{1/\theta}x^{-1}\,dx
+\int_1^{1/\theta}(\cos\theta x-1)x^{-1}\,dx
+\int_{1/\theta}^\infty\cos\theta x\,x^{-1}\,dx
\\[.75ex] & = \SMALL\log\frac{1}{\theta}+\int_0^1(\cos\sigma-1)\sigma^{-1}\,d\sigma
+ \int_1^\infty \cos\sigma\,\sigma^{-1}\,d\sigma + O(\theta),
\end{align*}
and
\begin{align*}
\SMALL\int_0^\infty \cos\theta x \,H_1(x)\,dx & \SMALL= \int_0^\infty H_1(x)\,dx
+\int_0^{1/\theta}(\cos\theta x-1)H_1(x)\,dx \\[.75ex] & \SMALL \qquad
+\int_{1/\theta}^\infty(\cos\theta x-1)H_1(x) \,dx
 = \int_0^\infty H_1(x)\,dx+ O(\theta^{q-1}).
\end{align*}
Hence
$I_C = c\log\frac{1}{\theta}+cc_H+O(\theta^{q-1})$.   Also,
\begin{align*}
I_S & = \SMALL\int_0^\infty \sin\theta x (1-G(x))\,dx
= c\int_1^\infty \sin\theta x\, x^{-1}\,dx
+ c\int_0^\infty \sin\theta x H_1(x)\,dx \\[.75ex] &\SMALL
= c\int_\theta^\infty \sin\sigma\,\sigma^{-1}\,d\sigma + O(\theta^{q-1})
= \frac{c\pi}{2} + O(\theta^{q-1}).
\end{align*}
Hence
\[
\SMALL
\Re(1-\lambda(\theta))=\frac{c\pi}{2}\theta(1+O(\theta^{q-1})), \quad
\Im\lambda(\theta)=c\theta(\log{\SMALL \frac{1}{\theta}})(1+c_H(\log\frac{1}{\theta})^{-1}+O((\log\frac{1}{\theta})^{-2})),
\]
and the result follows.
\end{proof}

\begin{lemma} \label{lem-second_one}
If $a=O(n^{1-\delta})$ for some $\delta>0$, then
\begin{itemize}
\item[(i)]
$(\log n)\int_{a/n}^{\pi} \cos n\theta\, T(\theta)\,d\theta
\ll a^{-1}$.
\item[(ii)] 
$(\log n)\int_0^{a/n} \cos n\theta\{T(\theta)-
(1-\lambda(\theta))^{-1}P\}\,d\theta \ll 
an^{-1}\log n$.
\end{itemize}
\end{lemma}

\begin{proof}
(i) Since $T(\theta)=(1-\lambda(\theta))^{-1}P +O(1)$, we have
 $\Re T(\theta)\ll \theta^{-1}(\log\theta)^{-2}$ for 
$\theta\in[0,\epsilon]$ and $T(\theta)\ll 1$ for $\theta\in[\epsilon,\pi]$.
The integral splits up into three parts as in Lemma~\ref{lem-GL}.   
As usual $I_1\ll n^{-1}$.  Next,
\begin{align*}
I_2 & \ll \int_{a/n}^{(a+\pi)/n} \Re T(\theta)\,d\theta
\ll \int_{a/n}^{(a+\pi)/n} \theta^{-1}(\log\theta)^{-2}\,d\theta \\
& \ll (\log(a/n))^{-1}- (\log((a+\pi)/n))^{-1} 
 = \frac{\log((a+\pi)/n)-\log(a/n)}{\log(a/n)\log((a+\pi)/n)}.
\end{align*}
Since $a=O(n^{1-\delta})$, we deduce that 
$I_2\ll \log(1+\pi/a)\,(\log n)^{-2} \ll a^{-1}(\log n)^{-2}$.

Finally, $R(\theta+h)-R(\theta)\ll h^{-1}\log h$ by Proposition~\ref{prop-H1}, 
so
\begin{align*}
\SMALL I_3 &\SMALL  \ll 
\int_{(a+\pi)/n}^\pi  
\|T(\theta)\|\,\|T(\theta-\pi/n)\|
\, \|R(\theta)-R(\theta-\pi/n)\|\,d\theta 
 \ll (\log n)n^{-1}(1+ A)
\end{align*}
where
\begin{align*}
\SMALL A &\SMALL  = \int_{(a+\pi)/n}^\epsilon
\|T(\theta)\|\,\|T(\theta-\pi/n)\| \,d\theta 
 \ll \int_{a/n}^\epsilon  (\theta\log\theta)^{-2}\,d\theta  \\[.75ex] &
\SMALL  \ll \int_{a/n}^b  (\theta\log\theta)^{-2}\,d\theta 
+  \int_b^\epsilon  (\theta\log\theta)^{-2}\,d\theta 
 \ll (\log b)^{-2}\int_{a/n}^b  \theta^{-2}\,d\theta 
+  (b\log b)^{-2}\int_b^\epsilon  1\,d\theta \\[.75ex] &
\SMALL  \ll (\log b)^{-2}n/a+ (b\log b)^{-2}.
\end{align*}
Taking $b=n^{-\frac12\delta}$ say, we obtain 
$I_3\ll (\log n)^{-1}(a^{-1}+n^{-(1-\delta)})\ll (\log n)^{-1}a^{-1}$.

\noindent(ii)  This is immediate since $T(\theta)=(1-\lambda(\theta))^{-1}P
+O(1)$.
\end{proof}

\begin{pfof}{Theorem~\ref{thm-second_one}}
We use Proposition~\ref{prop-second_one} to
estimate $\int_0^{a/n}\cos n\theta\,
\Re\{(1-\lambda(\theta))^{-1}\}\,d\theta$, discarding all terms that
are $O((\log n)^{-3})$, bearing in mind our eventual choice $a=n^\frac12$.

First we note that for $j\ge2$,
\begin{align} \label{eq-1}
\SMALL \int_0^{a/n}\theta^{-1}(\log\frac{1}{\theta})^{-j}\,d\theta 
= \frac{1}{j-1}(\log\frac{n}{a})^{-(j-1)}\ll (\log n)^{-(j-1)}.
\end{align}
In particular, taking $j=4$ disposes of the 
$O(\theta^{-1}(\log\frac{1}{\theta})^{-4})$ term.

Next we consider the $\theta^{-1}(\log\frac{1}{\theta})^{-3}$ term.
Using properties of oscillatory integrals,
\[
\SMALL \int_{1/n}^{a/n}\cos n\theta\, \theta^{-1}(\log\frac{1}{\theta})^{-3}\,d\theta 
= \int_1^a \cos\sigma\, \sigma^{-1}(\log{\SMALL \frac{n}{\sigma}})^{-3}\,d\sigma
\ll (\log n)^{-3},
\]
and
\[
\SMALL \int_0^{1/n} (\cos n\theta-1) \theta^{-1}(\log\frac{1}{\theta})^{-3}\,d\theta
=\int_0^1 (\cos \sigma-1) \sigma^{-1}(\log\frac{n}{\sigma})^{-3}\,d\sigma
\ll (\log n)^{-3}.
\]
Taking $j=3$ and $a=1$ in equation~\eqref{eq-1}, we deduce that
\[
\SMALL \int_0^{a/n} \cos n\theta\, \theta^{-1}(\log\frac{1}{\theta})^{-3}\,d\theta
=\frac12(\log n)^{-2}+O((\log n)^{-3}).
\]

To deal with the $\theta^{-1}(\log\frac{1}{\theta})^{-2}$ term,
we use the identity 
$\BIG \frac{\log n}{\log\frac{n}{\sigma}}=1+
\frac{\log \sigma}{\log\frac{n}{\sigma}}$.  So
\begin{align*}
\SMALL
& \SMALL \int_{1/n}^{a/n}\cos n\theta\, \theta^{-1}(\log\frac{1}{\theta})^{-2}\,d\theta 
 = (\log n)^{-2}\int_1^a \cos\sigma\, \sigma^{-1}\{\log n/\log{\SMALL \frac{n}{\sigma}}\}^2\,d\sigma
\\ 
& \SMALL
\quad  = (\log n)^{-2}\int_1^a \cos\sigma\, \sigma^{-1}\,d\sigma+O((\log n)^{-3})
= (\log n)^{-2}\int_1^\infty \cos\sigma\, \sigma^{-1}\,d\sigma+O((\log n)^{-3}),
\end{align*}
and
\begin{align*}
& \SMALL
\int_0^{1/n} (\cos n\theta-1) \theta^{-1}(\log\frac{1}{\theta})^{-2}\,d\theta
=(\log n)^{-2}\int_0^1 (\cos \sigma-1) \sigma^{-1}\{\log n/\log\frac{n}{\sigma}\}^2\,d\sigma \\ & \SMALL \quad
=
(\log n)^{-2}\int_0^1 (\cos \sigma-1) \sigma^{-1}\,d\sigma +O((\log n)^{-3}).
\end{align*}
Taking $j=2$ and $a=1$ in equation~\eqref{eq-1}, we deduce that
\[
\SMALL \int_0^{a/n} \cos n\theta\, \theta^{-1}(\log\frac{1}{\theta})^{-2}\,d\theta
=(\log n)^{-1}+A(\log n)^{-2}+O((\log n)^{-3}),
\]
where $A=\int_0^1 (\cos \sigma-1) \sigma^{-1}\,d\sigma +\int_1^\infty \cos\sigma\, \sigma^{-1}\,d\sigma$.

Combining these results, we obtain
\[
\SMALL
\frac{2}{\pi}\int_0^{a/n}\cos n\theta\,
\Re\{(1-\lambda(\theta))^{-1}\}\,d\theta= c^{-1} - c^{-1}\int_1^\infty H_1(x)\,dx
\,(\log n)^{-1}+O((\log n)^{-2}),
\]
which combined with Lemma~\ref{lem-second_one} (taking $a=n^{\frac12}$)
gives the required result.
\end{pfof}

\subsection{Convergence rates in the arcsine law}
\label{sec-arcsine}

As mentioned in Remark~\ref{rmk-DK}, a consequence of Theorem~\ref{thm-conv}
is that the Dynkin-Lamperti arcsine law for waiting times holds when
$\beta>\frac12$.    In fact, the arcsine law holds
for AFN maps for all $\beta$~\cite{Zweimuller00}.  
See also~\cite{Thaler98,ThalerZweimuller06} for more general transformations.
Here we show that our results on second order asymptotics 
yield a convergence rate.

For $x\in \bigcup_{j=0}^n f^{-j}Y$, $n\ge1$, let 
\[
Z_n(x)=\max\{0\leq j\leq n:f^jx\in Y\},
\]
denote the time of the last visit of the orbit of $x$ to 
$Y$ during the time interval $[0,n]$. 

Let $\zeta_{\beta}$ denote a random variable distributed according to
the $\mathcal{B}(1-\beta,\beta)$ distribution:
\[
\P(\zeta_\beta\leq t)=
d_\beta \int_0^t\frac{1}{u^{1-\beta}}\frac{1}{(1-u)^{\beta}}\,du,
\quad t\in [0,1],
\]
where $d_\beta=\frac{1}{\pi}\sin\beta\pi$.

\begin{cor} \label{cor-arcsine}   
Suppose that $\beta\in(\frac12,1)$ and that
$\mu(\varphi>n)= cn^{-\beta}+O(n^{-2\beta})$, where
$c>0$.   Let $\gamma=\min\{1-\beta,\beta-\frac12\}$.  

Let $\nu$ be an absolutely continuous probability measure on $Y$ with
density $g\in\mathcal{B}$.
Then there is a constant $C>0$ independent of $\nu$ such that
\[
\bigl|\nu\{{\SMALL\frac1n Z_n} \leq t\}-\P(\zeta_\beta\leq t)\bigr|
 \le  C \|g\|n^{-\gamma}.
\]
 \end{cor}

\begin{proof}
Following Thaler~\cite{Thaler00}, we notice that
\begin{align} \label{eq-Thaler}
\nu\{{\SMALL\frac1n} Z_n\leq t\}
&=\sum_{0\leq j\leq nt}\nu(f^{-j}\{\varphi>n-j\}),
\end{align}
\begin{align*}
\nu(f^{-j}\{\varphi>n-j\}) & =  \int_X 1_{\{\varphi>n-j\}}\circ f^j\, 1_Yg\,d\mu 
 = \int_X  1_{\{\varphi>n-j\}}\, L^j(1_Yg)\,d\mu \\
&   = \int_Y 1_{\{\varphi>n-j\}}\, T_jg\,d\mu.
\end{align*}
By Theorem~\ref{thm-second}, 
$T_jg = c^{-1}d_\beta j^{-(1-\beta)}
\bigl(1 + O (j^{-(1-\beta)}) + O(\|g\|j^{-(\beta-\frac12)})\bigr)$
uniformly on $Y$.  Combined with the assumption on $\mu(\varphi>n)$, we obtain
\begin{align*}
& \nu(f^{-j}\{\varphi>n-j\}) \\ 
& \qquad  =  d_\beta j^{-(1-\beta)} (n-j)^{-\beta}
\bigl(1+O (j^{-(1-\beta)}) + O(\|g\|j^{-(\beta-\frac12)})\bigr)
\bigl(1+ O((n-j)^{-\beta})\bigr).
\end{align*}
Since functions of the form $s^{-a}(n-s)^{-b}$ have only one turning point, 
replacing the sum in~\eqref{eq-Thaler} by an integral introduces only three
errors all of order $\|g\|n^{-1}$ and so
$\mu\{{\SMALL\frac1n} Z_n\leq t\} = d_\beta I + O(\|g\|n^{-1})$, where
\begin{align*}
I &   =
\int_0^{nt} s^{-(1-\beta)} (n-s)^{-\beta}
\bigl(1+O (s^{-(1-\beta)}) + O(\|g\|s^{-(\beta-\frac12)})\bigr)
\bigl(1+ O((n-s)^{-\beta})\bigr)\,ds \\
 & = \int_0^t u^{-(1-\beta)} (1-u)^{-\beta}\,du
+ O(n^{-(1-\beta)}) + O(\|g\| n^{-(\beta-\frac12)}),
\end{align*}
as required.
\end{proof}

\begin{rmk}
(a) The proof shows that for any $q\ge0$,
\[
\nu(\{{\SMALL\frac1n}Z_n \leq t\}) = 
\sum_{k=0}^q b_{\beta,k} \P(\zeta_{\beta,k+1}\le t) n^{-k(1-\beta)}+ 
O(n^{-(q+1)(1-\beta)})+O(\|g\|n^{-(\beta-\frac12)}).
\]
where $\zeta_{\beta,k}$ is the random variable with density
proportional to $u^{-k(1-\beta)}(1-u)^{-\beta}$
and $b_{\beta,k}=d_{\beta,k}/\int_0^1 u^{-(k+1)(1-\beta)}(1-u)^{-\beta}\,du$.
Here $b_{\beta,0}=1$ and $\zeta_{\beta,1}=\zeta_\beta$.

Thus for $\beta$ close to $1$, we obtain asymptotic expansions to
arbitrarily high order, and the error rate $n^{-\gamma}$ is optimal for
$\beta\ge\frac34$.

\noindent (b)   For $x\in X$, let $Y_n(x)=\min\{k>n: f^kx\in Y\}$.
Then $Y_k>n$ if and only if $Z_n\le k$ so that the arcsine law is equivalent
to strong distributional convergence of $\frac1n Y_n$ to $\zeta_\beta^{-1}$
(see for example~\cite{Thaler98}).  It is easily verified that the 
convergence rate in Corollary~\ref{cor-arcsine} holds also for $\frac1n Y_n$.
\end{rmk}

\section{Convergence results for $L^n$}
\label{sec-L}

Sections~\ref{sec-ORT} to~\ref{sec-second} were concerned with the analysis
of the sequence of renewal operators $T_n$ given by $T_nv=1_YL^n(1_Yv)$.   
An important issue is to study the iterates $L^n$ themselves.
In Subsection~\ref{sec-A}, we show how convergence on $Y$ 
implies convergence almost everywhere on $X$.
In Subsection~\ref{sec-B}, we consider observables not supported on $Y$.

\subsection{Convergence on $X$}
\label{sec-A}

Theorem~\ref{thm-conv} gives (uniform) convergence results on $Y$ for 
observables $v\in\mathcal{B}$.    Recall that $Y$ can be regarded as a first
return set for both the underlying system $f:X\to X$ and the tower map
$f_\Delta:\Delta\to\Delta$ introduced in Subsection~\ref{sec-tower}.
We now show that observables $v\in\mathcal{B}$ enjoy pointwise
convergence everywhere on $\Delta$ and almost everywhere on $X$.

\begin{prop} \label{prop-A}   Let $v\in L^1(\Delta)$, $w(n)\in\R$, $A\in\R$.
Suppose that $w(n)L_\Delta^n v \to A$ pointwise on $Y$.
Then $w(n)L_\Delta^n v \to A$ pointwise on $\Delta$.
\end{prop}

\begin{proof}   Let $p=(y,j)\in\Delta$.
Then $f_\Delta^{-j}p$ consists of the single preimage $(y,0)\cong y$, and
$w(n)(L_\Delta^n v)(p)=w(n)(L_\Delta^{n-j}v)(y)\to A$.
\end{proof}

Let $\pi:\Delta\to X$ be the projection $\pi(y,j)=f^jy$.
Let $1\le p,q\le\infty$, $\frac1p+\frac1q=1$.
Define $\pi^*:L^p(X)\to L^p(\Delta)$,
$\pi^* v=v\circ\pi$, and by duality define $\hat\pi:L^q(\Delta)\to L^q(X)$,
$\int_\Delta \hat\pi v\,w\,d\mu_X
=\int_X v\,\pi^*w\,d\mu_\Delta$.
As usual, $\|\pi^*\|_p=\|\hat\pi\|_q=1$ and
we have the standard properties 
$\hat\pi 1=1$, $\hat\pi \pi^* = I$, $\hat\pi L_\Delta = L\hat\pi$, 
$\hat \pi (1_{\pi^{-1}E}v)=1_E\hat\pi v$.

\begin{prop} \label{prop-pi}
Let $v\in L^1(\Delta)$, $w(n)\in\R$, $A\in\R$.
Suppose that $w(n) L_\Delta^n v \to A$ almost everywhere on $\Delta$.
Then $w(n) L^n \hat\pi v \to A$ almost everywhere on $X$.
\end{prop}

\begin{proof}
Suppose for contradiction that $E\subset X$ is a set of positive finite
measure such that everywhere on $E$, $w(n) L^n \hat\pi v$ fails to 
converge to $A$.   By assumption, $w(n) L_\Delta^n v\to A$ almost
everywhere on $\pi^{-1}E$.   By Egorov's Theorem, there is a subset $C\subset
\pi^{-1}E$ of positive measure such that $w(n) L_\Delta^n v\to A$ uniformly on $C$.
Indeed $\mu_\Delta(\pi^{-1}E\setminus C)$ is arbitrarily small, and since
$\pi$ has only countably many branches it follows from an $\epsilon/2^k$ 
argument that we can choose $C=\pi^{-1}E'$ where $E'$ is a positive measure
subset of $E$.

In particular, $\|1_{\pi^{-1}E'}(w(n) L_\Delta^n v-A)\|_{L^\infty(\Delta)}\to0$
and since $\hat\pi:L^\infty(\Delta)\to L^\infty(X)$ is bounded, 
$\|\hat\pi\{1_{\pi^{-1}E'}(w(n) L_\Delta^n v-A)\}\|_{L^\infty(X)}\to0$.
But $\hat\pi\{1_{\pi^{-1}E'}(w(n) L_\Delta^n v-A)\}=1_{E'}\hat\pi(w(n) L_\Delta^n v-A)=1_{E'}(w(n) L^n\hat\pi v-A)$ so we conclude that
$w(n) L^n\hat\pi v\to A$ on $E'$ which is the desired contradiction.
\end{proof}

\begin{cor}  \label{cor-A}
If $\beta\in(\frac12,1]$ and $v\in\mathcal{B}$, then
$\lim_{n\to\infty}m(n)n^{1-\beta}L^nv=d_\beta\int_Yv\,d\mu$
uniformly on $Y$ and almost everywhere on $X$.
\end{cor}

\begin{proof}  Since $v$ is supported on $Y$,
we can regard $v$ as an observable on $\Delta$ or on $X$ and $\hat\pi v=v$.
Also, $T_nv=1_YL^nv=1_YL_\Delta^nv$.
Theorem~\ref{thm-conv} immediately implies uniform convergence
on $Y$.  By Propositions~\ref{prop-A} and~\ref{prop-pi}, 
$m(n)n^{1-\beta}L_\Delta^nv$ converges pointwise on $\Delta$ and 
$m(n)n^{1-\beta}L^nv$ converges almost everywhere on $X$.
\end{proof}
 
\subsection{Convergence for general observables}
\label{sec-B}

In this subsection, we 
enlarge the class of observables so that they need not be supported on $Y$.
Define $X_k=f^{-k}Y \setminus \bigcup_{j=0}^{k-1}f^{-j}Y$.
Thus $z\in X_k$ if and only if $k\ge0$ is least such that $f^kz\in Y$.
(In particular, $X_0=Y$.)   

Given $v\in L^\infty(X)$, define $v_k=1_{X_k}v$.
Then  $L^kv_k$ is supported in $Y$ and $L^nv_k$ vanishes on $Y$
for all $n<k$.   If $n\ge k$, we have $1_YL^nv_k=T_{n-k}L^kv_k$.

Write $v\in\mathcal{B}(X)$ if $v\in L^1(X)$ and $L^kv_k\in\mathcal{B}$
for each $k\ge0$.  

\begin{thm} \label{thm-B}
Let $\beta\in(\frac12,1]$.  Suppose that
$v\in \mathcal{B}(X)$ and moreover that
$\sum_{n=0}^\infty \|L^nv_n\|<\infty$ and either (i)
$\|L^nv_n\|=o(n^{-1})$ or (ii) $\sum_{k=n}^\infty \|L^kv_k\|=o((m(n)n^{1-\beta})^{-1})$.  Then
$\lim_{n\to\infty}m(n)n^{1-\beta}L^n v=d_\beta \int_Xv\,d\mu$ uniformly on $Y$
and pointwise on $X$.
\end{thm}

\begin{proof}
Let $w(n)=d_\beta^{-1}m(n)n^{1-\beta}$.
First we show that
\begin{align} \label{eq-wn}
w(n)\sum_{n/2\le j\le n}w(n-j)^{-1}\|L^jv_j\|\to0.
\end{align}
In case (i), $\|L^nv_n\|=o(n^{-1})$, it follows from Karamata that
\begin{align*}
& w(n)\sum_{n/2\le j\le n}w(n-j)^{-1}\|L^jv_j\|
 \ll n^{-1}w(n) \sum_{n/2\le j\le n}w(n-j)^{-1}(j\|L^jv_j\|)
\\ & \qquad\qquad  \le n^{-1}w(n)\Bigl(\sum_{n/2\le j\le n}w(n-j)^{-1}\Bigr)\max_{j\ge n/2}j\|L^jv_j\|
 \ll \max_{j\ge n/2}j\|L^jv_j\|\to0.
\end{align*}
In case (ii),
\begin{align*}
 w(n)\sum_{n/2\le j\le n}w(n-j)^{-1}\|L^jv_j\|\ll
& w(n)\sum_{n/2\le j\le n}\|L^jv_j\|\ll
 w(n/2)\sum_{j\ge n/2}\|L^jv_j\|\to0.
\end{align*}
Hence~\eqref{eq-wn} is verified in both cases.

Next we prove uniform convergence on $Y$.  Let
 $c_{j,n}=\frac{w(n)}{w(n-j)}-1$.
By Theorem~\ref{thm-conv},
$T_n=w(n)^{-1}P+S_n$ where $S_n=o(w(n)^{-1})$.  Hence on $Y$,
\begin{align*}
& w(n)L^nv-{\SMALL\int} v   =w(n)\sum_{j=0}^n T_{n-j}L^jv_j-\sum_{j=0}^\infty{\SMALL\int} v_j
\\ & \qquad  =
w(n)\sum_{j=0}^n w(n-j)^{-1} {\SMALL\int} L^jv_j
-\sum_{j=0}^n{\SMALL\int} L^jv_j
+w(n)\sum_{j=0}^n S_{n-j}L^jv_j
-\sum_{j>n}{\SMALL\int} v_j
\\ & \qquad  =
\sum_{j=0}^n c_{j,n} {\SMALL\int} L^jv_j
+w(n)\sum_{j=0}^n S_{n-j}L^jv_j
-\sum_{j>n}{\SMALL\int} v_j,
\end{align*}
and so
\begin{align*}
|w(n)L^nv-{\SMALL\int} v| \le
\sum_{j=0}^n |c_{j,n}||{\SMALL\int}L^jv_j|
+w(n)\sum_{j=0}^n \|S_{n-j}\|\|L^jv_j\|
+\Bigl|\sum_{j>n}{\SMALL\int} v_j\Bigr|.
\end{align*}

It is immediate that the third term converges to zero.
Write $\|S_n\|=w(n)^{-1}a_n$ where
$a_n=o(1)$.  Then the second term satisfies
\begin{align*}
 w(n)\sum_{j=0}^n \|S_{n-j}\|\|L^jv_j\|
& =
w(n)\sum_{0\le j<n/2} w(n-j)^{-1}a_{n-j}\|L^jv_j\|
\\ & \qquad \qquad \qquad +w(n)\sum_{n/2\le j\le n} w(n-j)^{-1}a_{n-j}\|L^jv_j\| \\
& \ll \sum_{0\le j<n/2} a_{n-j}\|L^jv_j\|+
w(n)\sum_{n/2\le j\le n} w(n-j)^{-1}\|L^jv_j\| \\
& \le
\Bigl(\sum_{k=0}^\infty \|L^kv_k\|\Bigr) \max_{j\ge n/2} a_j +
 w(n)\sum_{n/2\le j\le n} w(n-j)^{-1}\|L^jv_j\| \to0
\end{align*}
by~\eqref{eq-wn}, summability of $\|L^kv_k\|$ and the definition of $a_n$.

The first term satisfies
\begin{align*}
\sum_{j=0}^n |c_{j,n}||{\SMALL\int}L^jv_j| &
 \le
\sum_{0\le j<n/2} |c_{j,n}|\|L^jv_j\|+ \sum_{n/2\le j\le n} |c_{j,n}|\|L^jv_j\|
\\ &  \ll \sum_{0\le j<n/2} |c_{j,n}|\|L^jv_j\|+ w(n)\sum_{n/2\le j\le n} w(n-j)^{-1} \|L^jv_j\|.
\end{align*}
Again $w(n)\sum_{n/2\le j\le n}  w(n-j)^{-1}\|L^jv_j\|\to0$ by~\eqref{eq-wn}.  Since
$\lim_{n\to\infty}c_{j,n}=0$ for each fixed $j$ and
$\max_{j,n\,:\,0\le j<n/2}|c_{j,n}|<\infty$, it follows from summability of
$L^nv_n$ that $\sum_{0\le j<n/2} |c_{j,n}|\|L^jv_j\|\to0$.
This completes the proof of uniform convergence on $Y$.

To prove pointwise convergence on $X$, define $u=\pi^*v:\Delta\to\R$
and note that $\hat\pi u=v$.   Also define $\Delta_k=f_\Delta^{-k}Y\setminus
\bigcup_{j=0}^{k-1}f_\Delta^{-j}Y$.   Then $\Delta_0=Y$
and for $k\ge1$, $\Delta_k$ consists of those points
$(y,j)\in\Delta$ with $j=\varphi(y)-k>0$.
Note also that $\pi^{-1}X_k\subset\Delta_k$ (since $(y,j)\in\pi^{-1}X_k$
if and only if $f^jy\in X_k$, that is $\varphi(y)=j+k$).
In particular, $u_k=\pi^*v_k$ is supported in $\Delta_k$.  
Hence $L_\Delta^ku_k$ is supported in $Y$ and 
$L_\Delta^ku_k=\hat\pi L_\Delta^ku_k=L^k\hat\pi u_k=L^kv_k$.
In particular, $L_\Delta^ku_k$ inherits the assumptions on $L^kv_k$, and the argument above
shows that $w(n)L_\Delta^n u\to \int_\Delta u\,d\mu_\Delta$ uniformly
on $Y$.   By Proposition~\ref{prop-A}, pointwise convergence extends 
to $\Delta$.   By Proposition~\ref{prop-pi}, pointwise convergence for $u$
drops down to pointwise convergence for $v$.
\end{proof}

Note that Theorem~\ref{thm-B} includes the case where $v$ is supported
on $\bigcup_{j=0}^k X_j$ for some $k$, and hence significantly extends 
Theorem~\ref{thm-conv}.

In the next result, we extend Theorem~\ref{thm-liminf}, and
we drop the assumptions on $L^nv_n$ in Theorem~\ref{thm-B}.

\begin{prop} \label{prop-B}
Let $\beta\in(0,1]$  and $v\in \mathcal{B}(X)$.
\begin{itemize}
\item[(a)]   For each $y\in Y$, there is a zero density set $E\subset\N$ such
that 
$\lim_{n\to\infty,\;n\not\in E}m(n)n^{1-\beta}(L^n v)(y)= d_\beta\int_Xv\,d\mu$.
\item[(b)]   If $v\ge0$,
then $\liminf_{n\to\infty}m(n)n^{1-\beta} L^n v=d_\beta\int_Xv\,d\mu$
pointwise on $Y$.
\end{itemize}
\end{prop}

\begin{proof}
Let $w(n)=d_\beta^{-1}m(n)n^{1-\beta}$.
By the argument in the proof of 
Theorem~\ref{thm-liminf}, it suffices to prove the $\ge$ inequality in part (b).
Let $v\ge0$ and define $v(k)=\sum_{j=0}^k v_j$.
By Theorem~\ref{thm-liminf}(b),
$\liminf_{n\to\infty} w(n)1_YL^nv\ge \liminf_{n\to\infty} w(n)1_YL^nv(k)
=\sum_{j=0}^k \liminf_{n\to\infty} w(n)1_YL^nv_j 
=\sum_{j=0}^k \int v_j= \int v(k)$.   Since $k$ is arbitrary, 
$\liminf_{n\to\infty}w(n)1_YL^nv\ge \int v$ as required.~
\end{proof}

\subsection{Second order asymptotics on $X$}
Under the assumptions of Theorem~\ref{thm-second}, we can
investigate second order asymptotics in Theorem~\ref{thm-B}.
For example, we have the following:

\begin{thm} \label{thm-secondB}
Suppose that $\beta\in(\frac12,1)$ and that
$\mu(\varphi>n)= cn^{-\beta}+O(n^{-2\beta})$ where $c>0$.
Suppose further that $v\in \mathcal{B}(X)$ and that
(i) $\|L^kv_k\|=O(k^{-p})$, and (ii) $\int L^kv_k\,d\mu=O(k^{-q})$,
where $p>\max\{\frac32-\beta,\beta\}$ and $q>1$.   Then
\[
n^{1-\beta}L^nv=c^{-1}d_\beta {\SMALL\int}_X v\,d\mu + O(n^{-\gamma}\|v\|)
\enspace\text{uniformly on $Y$},
\]
where $\gamma=\min\{1-\beta,\beta-\frac12,q-1\}$ if $p>1$, and
$\gamma=\min\{p-\beta,p+\beta-\frac32,q-1\}$ if $p>\max\{\frac32-\beta,\beta\}$.
\end{thm}

\begin{proof}
We assume without loss that $q\in(1,2)$.
In the notation of the proof of Theorem~\ref{thm-B},
\begin{align*}
cd_\beta^{-1}n^{1-\beta}L^nv-{\SMALL\int} v &
 \ll \sum_{j=0}^n |c_{j,n}||{\SMALL\int} L^jv_j|+n^{1-\beta}\sum_{j=0}^n \|S_{n-j}\|\|L^jv_j\|+\Bigl|\sum_{j>n}{\SMALL\int}L^jv_j\Bigr|.
\end{align*}
It is immediate that the third term is $O(n^{-(q-1)})$.
Since $c_{j,n}=(n/(n-j))^{1-\beta}-1=(1+(j/(n-j)))^{1-\beta}-1\ll j/(n-j)$,
the first term satisfies
\begin{align*}
 \sum_{j=0}^n |c_{j,n}||{\SMALL\int}L^j v_j| & \le
 \sum_{0\le j<n/2} |c_{j,n}||{\SMALL\int} L^jv_j|+
 \sum_{n/2\le j\le n} |c_{j,n}||{\SMALL\int} L^jv_j|
\\ & \ll \sum_{0\le j<n/2} (j/(n-j)) j^{-q}  + n^{-q}n^{1-\beta}\sum_{n/2\le j\le n} (n-j)^{-(1-\beta)}
\\ & \le  n^{-1}\sum_{1\le j<n/2} j^{1-q} + n^{1-\beta-q}\sum_{j=1}^n j^{-(1-\beta)}
 \ll n^{-(q-1)}.
\end{align*}

By Theorem~\ref{thm-second}, $\|S_n\|=O(n^{-(1-\beta+\gamma_1)})$ where
$\gamma_1=\min\{1-\beta,\beta-\frac12\}$.  Hence the second term satisfies
\begin{align*}
n^{1-\beta}\sum_{j=0}^n \|S_{n-j}\|\|L^jv_j\|\ll n^{1-\beta}\,\{(n^{-(1-\beta+\gamma_1)})\star (n^{-p})\}\ll\begin{cases} n^{-\gamma_1}, & p>1 \\ n^{-(\gamma_1+p-1)}, & p<1 \end{cases}
\end{align*}
as required.
\end{proof}

\section{Examples}
\label{sec-examples}

In this section, we apply our results to specific examples.
In Subsection~\ref{sec-H}, we describe a method for verifying
our functional-analytic hypotheses (H1), (H2), that suffices for our purposes.
In Subsection~\ref{sec-GM}, we consider situations where the first 
return map $F:Y\to Y$ is Gibbs-Markov.  This includes the nonuniformly
expanding maps studied by Thaler~\cite{Thaler95} and parabolic rational
maps of the complex plane~\cite{ADU93}.    In Subsection~\ref{sec-AFN}, we 
consider the full class of AFN maps~\cite{Zweimuller98}.
In Subsection~\ref{sec-PM}, we specialise to the case 
of Pomeau-Manneville intermittency maps~\eqref{eq-LSV}.

\subsection{Verification of hypotheses (H1) and (H2)}
\label{sec-H}

In our examples, (H1) can be verified in the process of 
verifying (H2), so we focus on the latter.
The standard approach (cf.~\cite[Lemma~6.7]{Gouezel04} and~\cite[Section~5]{Sarig02}) to (H2) proceeds via the following result.

\begin{prop} \label{prop-H2}
Suppose that $F:Y\to Y$ is ergodic.
Assume that 
(1)   $R(z):\mathcal{B}\to\mathcal{B}$
has essential spectral radius strictly less
than $1$ for every $z\in\bar\D$.
(2) For each $\theta\in(0,2\pi)$, there are no nontrivial $L^2$ solutions
to the equation $v\circ F=e^{i\theta\varphi}v$ $a.e.$
Then (H2) is satisfied.
\end{prop}

\begin{proof}
By (1), it suffices to consider generalized eigenfunctions $v\in L^2$
for the operator $R(z)$.
Suppose that $R(z)v=v$ where $v\in L^2$ is nonzero.  
Write $z=\rho e^{i\theta}$, $\rho\in[0,1]$, $\theta\in[0,2\pi)$.  Then
$|v|_2=|R(z)v|_2=|R(\rho^\varphi e^{i\theta\varphi}v)|_2=|\rho^\varphi v|_2\le 
|\rho^\varphi|_\infty|v|_2
\le \rho|v|_2$,
so $\rho=1$ and $R(e^{i\theta})v=v$.
The $L^2$ adjoint of $U=R(e^{i\theta})$ is the operator 
$U^*v= e^{-i\theta\varphi}v\circ F$ and an elementary calculation shows that
$|U^*v-v|_2^2=|Uv|_2^2-|v|_2^2=0$.    
Hence $v\circ F=e^{i\theta\varphi}v$.
By (2), $\theta=0$.  Hence we have established (H2)(ii).

When $z=1$, the eigenvalue $1$ is isolated in the spectrum
by (1) and the eigenvalue is simple by ergodicity of $F$,
so (H2)(i) is valid.
\end{proof}

\begin{defn} Suppose that $Y$ is a topological space, that
$(Y,\mu)$ is a probability space,
and that $F:Y\to Y$ is a measure preserving transformation.
Let $\varphi:Y\to\R$ be a measurable map.   We say that $(F,\varphi)$ satisfies
property (*) if for every $\theta\in[0,2\pi]$ and every measurable
solution $v:Y\to S^1$ to the equation $v\circ F=e^{i\theta\varphi}v$ $a.e.$,
there exists an open set $U\subset Y$ such that $v$ is constant almost 
everywhere on $U$.
\end{defn}

\begin{lemma} \label{lem-H2}
Suppose that $f$ is topologically mixing with first return
map $F=f^\varphi:Y\to Y$.  Assume that $(F,\varphi)$ satisfies property (*).
Then there are no nontrivial measurable solutions $v:Y\to\C$ to the equation
$v\circ F=e^{i\theta\varphi}v$ for all $\theta\in(0,2\pi)$.
\end{lemma}

\begin{proof}
If $v$ is a nontrivial solution, then by (*) there is an open set $U\subset Y$
on which $v$ is almost everywhere constant.
Now we follow the second half of the proof of~\cite[Lemma~6.7]{Gouezel04}.
Since $f$ is topologically mixing,
there exists $N\ge1$ such that $f^nU \cap U\neq\emptyset$
for all $n\ge N$.  In particular, for each $n\ge N$ we can choose 
$y\in U$ such that $f^ny\in U$ and $v(y)=v(f^ny)\neq0$.

Let $0=k_0<k_1<\dots<k_p=n$ be the successive return times of $y$ to $Y$.
Then $T^ny=F^py$ and $n=\varphi_p(y)=\sum_{j=0}^{p-1}\varphi(F^jy)$.
Hence
\[
e^{i\theta n}=\prod_{j=0}^{p-1}e^{i\theta \varphi(F^jy)}=
\prod_{j=0}^{p-1}\frac{v(F^{j+1}y)}{v(f^jy)}=\frac{v(f^ny)}{v(y)}=1.
\]
Taking $n=N$ and $n=N+1$, we deduce that $e^{i\theta}=1$ which is
a contradiction.
\end{proof}

\subsection{Maps with Gibbs-Markov first return maps}
\label{sec-GM}

A large class of examples covered by our methods are those with first return 
maps that are Gibbs-Markov.    This includes parabolic rational
maps of the complex plane (Aaronson {\em et al}~\cite{ADU93}) and Thaler's class of interval
maps with indifferent fixed points~\cite{Thaler95} (in particular the 
family~\eqref{eq-LSV}).    

We recall the key definitions~\cite{Aaronson}.
Let $(X,\mu)$ be a Lebesgue space with 
countable measurable partition $\alpha_X$.
Let $f:X\to X$ be an ergodic, conservative, measure-preserving, Markov map transforming
each partition element bijectively onto a union of partition elements.
Recall that $f$ is {\em topologically mixing} if for
all $a,b\in\alpha_X$ there exists $N\ge1$ such that $b\subset f^na$ for all
$n\ge N$.

Let $Y$ be a union of partition elements with $\mu(Y)\in(0,\infty)$.
Define the first return time $\varphi:Y\to\R$ and first return map
$F=f^\varphi:Y\to Y$.  Let $\alpha$ be the partition of $Y$ consisting
of nonempty cylinders of the form $a\cap(\bigcap_{j=1}^{n-1} T^{-j}\xi_j)\cap T^{-n}\alpha$ where $a,\xi_j\in\alpha_X$, and $a\subset Y$, $\xi_j\subset X\setminus Y$.
Fix $\tau\in(0,1)$ and define $d_\tau(x,y)=\tau^{s(x,y)}$
where the {\em separation time} $s(x,y)$ is the greatest integer $n\ge0$
such that $F^nx$ and $F^ny$ lie in the same partition element in $\alpha$.
It is assumed that the partition $\alpha$ separates orbits of $F$, so
$s(x,y)$ is finite for all $x\neq y$.   Then $d_\tau$ is a metric.
Let $\Lip(Y)$ be the Banach space of $d_\tau$-Lipschitz
functions $v:Y\to\R$ with norm $\|v\|=|v|_\infty+\Lip(v)$. 

Define the potential function $p=\log\frac{d\mu}{d\mu\circ F}:Y\to\R$.
We require that $p$ is uniformly piecewise Lipschitz: that is,
$p|_a$ is $d_\tau$-Lipschitz for each $a\in\alpha$ and
the Lipschitz constants can be chosen independent of $a$.   
We also require the big images condition $\inf_a \mu(Fa)>0$.
A {\em Gibbs-Markov map} is a Markov map with uniformly piecewise Lipschitz potential
and satisfying the big images property.

\begin{prop} \label{prop-GM}   Suppose that $(X,\mu)$ is a Lebesgue space,
that $f:X\to X$ is an ergodic, conservative, measure preserving, topologically
mixing, Markov map, and that $Y\subset X$ is a union of partition elements
with $\mu(Y)\in(0,\infty)$.   Suppose further that the first return map
$F=f^\varphi:Y\to Y$ is Gibbs-Markov.  Then
the Banach space $\mathcal{B}=\Lip(Y)$ satisfies hypotheses (H1) and (H2).
\end{prop}

\begin{proof}
Since the details can be found in~\cite{Gouezel04,Sarig02}, we only sketch the
argument.
By equation~(8) in the proof of~\cite[Lemma~6.7]{Gouezel04}, 
there is a constant $C>0$ such that
$\|R(z)^n v\|\le C(|v|_\infty + \tau^n\|v\|)$ for all $v\in\mathcal{B}$,
$z\in\bar\D$,
$n\ge1$.     Since the unit ball in $\mathcal{B}$ is compact in $L^\infty$,
it follows from~\cite{Hennion93} that the essential spectral radius of
$R(z)$ is at most $\tau$ establishing property (1) of 
Proposition~\ref{prop-H2}.
Property (*) follows from~\cite[Theorem~3.1]{AaronsonDenker01}: 
measurable solutions
$v$ are constant almost everywhere on each partition element of $\alpha$
(even $F\alpha$).   Hence property (2) of Proposition~\ref{prop-H2}
follows from Lemma~\ref{lem-H2}.    This completes the verification of (H2).
(H1) is established during the proof of~\cite[Lemma~6.7]{Gouezel04},
see~\cite[Lemma~6.4]{Gouezel04}.   
\end{proof}

\begin{cor} \label{cor-GM}
In the setting of Proposition~\ref{prop-GM}, if in addition $\mu(X)=\infty$
and $\mu(y\in Y:\varphi(y)>n)$ is regularly varying with index $\beta\in(0,1]$,
then our main results
(including Theorems~\ref{thm-conv},~\ref{thm-bound},~\ref{thm-liminf})
apply.~\qed
\end{cor}

\begin{ex}[Parabolic rational maps of the complex plane]
Let $f:\bar\C\to\bar \C$ be a rational map of the Riemann sphere with spherical metric $d$.   
A period $k$ point $z\in\C$ is {\em rationally indifferent} if $(f^k)'(z)$ is a root of unity.
The map $f$ is {\em parabolic}
if $J$ contains no critical points and contains at least one rationally indifferent periodic point~\cite{DenkerUrbanski91}

Aaronson {\em et al.}~\cite[Section~9]{ADU93} establish a number of properties of parabolic rational maps.  Such maps are topologically mixing, conservative and exact with
respect to Lebesgue measure and possess a $\sigma$-finite invariant measure
$\mu$ equivalent to Lebesgue.  Moreover, there is a Gibbs-Markov first return map $F=f^\varphi:Y\to Y$ where $\mu(Y)\in(0,\infty)$.   Criteria are given for $\mu(X)$ to be
finite or infinite, and in the infinite case it is shown that
$\mu(\varphi>n)\sim Cn^{-\beta}$ where $C>0$ and $\beta\in(0,1]$.
For any $\eta>0$, it is possible to choose $\tau\in(0,1)$ so that 
$C^\eta(Y)\subset \Lip(Y)$.
By Corollary~\ref{cor-GM}, our main results apply to H\"older observables supported on $Y$.
\end{ex}

\begin{ex}[Thaler maps] \label{ex-Thaler}
Thaler~\cite{Thaler95} considers a class of topologically mixing
one-dimensional maps $f:X\to X$, $X=[0,1]$ for which there is a countable measurable partition $\{B(k):k\in I\}$ consisting of intervals,
and a nonempty finite set $J\subset I$ such that each $B(j)$, $j\in J$,
contains an indifferent fixed point $x_j$ with $f'(x_j)=1$.  It is required that
\begin{itemize}
\item[(1)] $f|_{B(k)}$ is twice differentiable and $\overline{fB(k)}=[0,1]$
for all $k$.
\item[(2)] $|f'|\ge \rho(\epsilon)>1$ on $\bigcup_{k\in I}B(k)\setminus
\bigcup_{j\in J}(x_j-\epsilon,x_j+\epsilon)$ for each $\epsilon>0$.
\item[(3)] For each $j\in J$ there exists $\eta>0$ such that $f'$ is
decreasing on $(x_j-\eta,x_j)\cap B(j)$ and increasing on
$(x_j,x_j+\eta)\cap B(j)$.
\item[(4)] $f''/(f')^2$ is bounded on $\bigcup_{k\in I} B(k)$.
\end{itemize}
Such a map $f$ is conservative and exact with respect to Lebesgue measure,
and admits an infinite $\sigma$-finite invariant measure $\mu$ equivalent
to Lebesgue measure $m$.
The density $h=d\mu/dm$ is H\"older, bounded below, and bounded above on compact subsets of $X'$.
As a special case of the construction in Zweim\"uller~\cite{Zweimuller00}, $f$ has a Gibbs-Markov first return map $F=f^\varphi:Y\to Y$, where $\mu(Y)\in(0,\infty)$.
Furthermore, for every compact set $C\subset X'$
(the complement of the indifferent fixed points), the first return set $Y$ 
can be chosen to contain $C$.
As shown in~\cite{Thaler95}, a sufficient condition for regularly varying return tail probabilities
is that $f$ has a ``good'' asymptotic
expansion near each indifferent fixed point.   For example, it suffices
that $f(x)=x+a_j|x-x_j|^{p_j+1}+o(|x-x_j|^{p_j+1})$ as $x\to x_j$ for each $j\in J$,
where $a_j\neq0$, $p_j\ge1$ and $p=\max_jp_j>1$.   In this case $\mu(\varphi>n)\sim Cn^{-1/p}$ where $C>0$.

To summarise, suppose that $f$ is a Thaler map and
$\mu(\varphi>n)=\ell(n)n^{-\beta}$ is regularly varying with 
$\beta\in(0,1]$.   Set $w(n)=d_\beta^{-1}m(n)n^{1-\beta}$.
By Corollary~\ref{cor-GM}, our main results apply to H\"older continuous
observables $v:X\to\R$ supported on a compact subset of $X'$.  In particular,
if $\beta\in(\frac12,1]$, we 
obtain uniform convergence on compact subsets of $X'$.

\begin{thm} \label{thm-thaler}
Suppose that $f:X\to X$ is a Thaler map with regularly varying tails, 
$\beta\in(\frac12,1]$.   
Then $\lim_{n\to\infty}w(n)L^nv=\int_Xv\,d\mu$ uniformly on compact
subsets of $X'$ for all $v$ of the form $v=u/h$ where $u$ is 
Riemann integrable on $X$.
\end{thm}

\begin{proof}
Suppose first that $u$ is H\"older.
Fix a first return set $Y$ chosen so that $f(Y)=X$.
Define the sets $X_k$ as in Subsection~\ref{sec-B} and write
$v=\sum v_k$ where $v_k=v|_{X_k}$.
We claim that $\|L^kv_k\| \ll \mu(\varphi=k+1)\|u\|_{C^\eta}$.
Hence $\sum_{k\ge n}\|L^kv_k\|\le\mu(\varphi>n)\|u\|_{C^\eta}
=O(\ell(n)n^{-\beta})=o(w(n))^{-1})$  since $\beta>\frac12$, verifying
condition (ii) in Theorem~\ref{thm-B}.
Also $u$ is Lebesgue integrable, so $v\in L^1(X,\mu)$.  
It follows from Theorem~\ref{thm-B}
that $\lim_{n\to\infty}w(n)L^nv=\int_Xv\,d\mu$
uniformly on $Y$.    

To prove the claim, let $\tilde L$ and $\tilde R$ denote the (unnormalised) transfer operators for $f$ and $F$ respectively corresponding to Lebesgue measure $m$ (rather than $\mu$).
In particular, $L=h^{-1}\tilde Lh$.
Every point in $X_k$ has a preimage in $Y$.   To simplify notation, suppose
that such preimages are unique (otherwise specify one of the preimages
and omit the other preimages in the argument below).
Write $(\tilde Lv)(x)=\sum_{fx'=x}g(x')v(x')$
and $(\tilde L^nv)(x)=\sum_{f^nx'=x}g_n(x')v(x')$ where
$g_n(x)=g(x)g(fx)\cdots g(f^{n-1}x)$.
Similarly, write $(\tilde Rv)(y)=\sum_{Fy'=y}G(y')v(y')$.
Let $u_k=u|_{X_k}$.  Then 
\begin{align*}
(\tilde L^ku_k)(x) & =\sum_{f^kx'=x}g_k(x')1_{X_k}(x')u(x')
=\sum_{y\in\{\varphi =k+1\},\;f^{k+1}y=x}g(y)^{-1}g_{k+1}(y)u(fy) \\ &
=\sum_{Fy=x}g(y)^{-1}1_{\{\varphi=k+1\}}(y)G(y)u(fy)
=\sum_a g(y_a)^{-1}G(y_a)u(fy_a),
\end{align*}
where the summation is over those $a$ with $\varphi|_a=k+1$,
and $y_a$ is the unique point in $a$ such that $Fy_a=x$.
For Gibbs-Markov maps it is standard that $\|1_aG\|\ll \mu(a)$.   
For the systems in~\cite{Thaler95}, $f'$ is bounded and $f$ is uniformly
expanding on $Y$ so that
$\|1_{\{\varphi=k+1\}}(u\circ f)\|\le C\|u_k\|_{C^\eta(X_k)}$ and
$\|1_Yg^{-1}\|<\infty$.
Hence $\|\tilde L^ku_k\|\ll \sum_a  
\mu(a)\|u\|_{C^\eta(X_k)}=\mu(\varphi=k+1)\|u\|_{C^\eta(X_k)}$, and so
$\|L^kv_k\|=\|h^{-1}\tilde L^ku_k\|\ll \|\tilde L^ku_k\|\ll\mu(\varphi=k+1)\|u\|_{C^\eta(X_k)}$.  This completes the proof of the claim.

Finally, if $v=u/h$ where $u$ is Riemann integrable, then
we can approximate $u$ from above and below by 
H\"older functions $u^\pm$.  In particular,
$w(n)\tilde L^n u^\pm=w(n)hL^n (u^\pm/h)\to h\int_Xu^\pm\,dm$ uniformly on
compact subsets of $X'$.
By positivity of the transfer operator, $\tilde L^n u^-\le \tilde L^nu\le\tilde L^nu^+$ and so $h\int_Xu^-\,dm\le \liminf w(n)\tilde L^nu
\le \limsup w(n)\tilde L^nu\le h\int_Xu^+\,dm$.
Since $\int_X(u^+-u^-)\,dm$ 
can be made arbitrarily small, $w(n)L^nv=w(n)h^{-1}\tilde L^nu\to \int_X u\,dm=\int_Xv\,d\mu$ uniformly on compact subsets of $X'$.~
\end{proof}

\begin{rmk}   \label{rmk-thaler}
Suppose that $v\in L^1(X,\mu)$ and that $v$ is H\"older on each $X_k$.  
As is evident from the proof, 
it suffices that 
$\sum_{k\ge n} \mu(\varphi=k)\|1_{X_k}(vh)\|_{C^\eta(X_k)}=o(w(n)^{-1})$ for some $\eta>0$.  Alternatively, using condition~(i) in
Theorem~\ref{thm-B}, it suffices that $\mu(\varphi=n)\|1_{X_n}(vh)\|_{C^\eta(X_n)}$ is summable and $o(n^{-1})$.
%
\end{rmk}
\end{ex}

\subsection{AFN maps}
\label{sec-AFN}

Zweim\"uller~\cite{Zweimuller98,Zweimuller00} studied a class
of non-Markovian interval maps $f:X\to X$, $X=[0,1]$,
with indifferent fixed points.   It is assumed that there is a measurable
partition $\xi$
of $X$ into open intervals such that $f$ is $C^2$ and strictly monotone on each $Z\subset\xi$, and such that the following conditions are satisfied:
\begin{itemize}
\item[(A)] \emph{Adler's condition:} $f''/(f')^2$ is bounded on $\bigcup_{Z\in\xi}Z$,
\item[(F)] \emph{Finite images:} $\{TZ:Z\in\xi\}$ is finite.
\item[(N)] \emph{Nonuniform expansion:} There is a finite set $\zeta\subset\xi$
such that each interval $Z\in\zeta$ has an indifferent fixed point $x_Z$ at one of its endpoints (so $fx_Z=x_Z$ and $f'(x_Z)=1$) such that $f$ has a $C^1$ extension to $Z\cup x_Z$ and $T'$ is increasing (resp.\ decreasing) on $Z$ if $x_Z$ is the left (resp.\ right) end point of $Z$.  Moreover, $|f'|\ge\rho(\epsilon)>1$ on $X\setminus\bigcup_{Z\in\zeta}((x_Z-\epsilon,x_Z+\epsilon)\cap Z)$ for
each $\epsilon>0$.
\end{itemize}
Such a map is called an {\em AFN map}. A Thaler map (Example~\ref{ex-Thaler}) is an AFN map with full branches.    If condition (N) is replaced by
\begin{itemize}
\item[(U)] \emph{Uniform expansion:} $|f'|\ge \rho>1$ on $\bigcup_{Z\in\xi}Z$,
\end{itemize}
then $f$ is called an AFU map.

By the spectral decomposition theorem
in~\cite{Zweimuller98}, any AFN map decomposes into basic sets that are topologically mixing up to a finite cycle.   From now on we suppose that $f:X\to X$ is a topologically
mixing AFN map.  Such a map is conservative and exact with respect to Lebesgue measure $m$, and admits an equivalent $\sigma$-finite invariant measure $\mu$.
The measure is infinite if and only if $X$ includes an indifferent fixed point,
and we suppose that this is the case.
Let $X'\subset X$ denote the complement of the indifferent fixed points.   
The density $h=d\mu/dm$ is of bounded variation, bounded below, and bounded above on compact subsets of $X'$.

\begin{prop} \label{prop-AFN}   If $f:X\to X$ is a topologically mixing AFN map
with $\mu(X)=\infty$,
and $C$ is a compact subset of $X'$, then there exists a first return set
$Y$ with $\mu(Y)\in(0,\infty)$ such that $Y$ contains $C$, and such that
the first return map $F=f^\varphi:Y\to Y$ is AFU.   Moreover, the
Banach space $\mathcal{B}=BV(Y)$ consisting of bounded variation functions 
on $Y$ satisfies hypotheses (H1) and (H2).
\end{prop}

\begin{proof}
By~\cite[Lemma~8]{Zweimuller00}, the first return map $F$ is AFU.
By~\cite{Rychlik83} and~\cite[Appendix]{Zweimuller98}), 
$\mathcal{B}=BV(Y)$ is a suitable Banach space.
In particular, $R(1):\mathcal{B}\to\mathcal{B}$ has essential spectral 
radius less than $1$.   The argument in~\cite{Rychlik83} is extended
by~\cite[Proposition~4]{ADSZ04} who show that 
there exist constants $C>0$, $\tau<(0,1)$ such that
$\|R(e^{i\theta})^nv\|\le C(|v|_1 + \tau^n\|v\|)$ for all $v\in\mathcal{B}$,
$\theta\in\R$, $n\ge1$.   
It is easy to extend this argument to cover $R(z)$ for all $z\in\bar\D$.   
(It should be noted that in our setting, the proof in~\cite{ADSZ04} is greatly
simplified since in~\cite{ADSZ04}
$\varphi$ is not assumed to be locally constant and $F$ is not required to
satisfy Adler's condition or finite images.)
Since the unit ball in $\mathcal{B}$ is compact in $L^1$,
property (1) of Proposition~\ref{prop-H2} again follows from~\cite{Hennion93}.
Property (*) follows 
from~\cite[Theorems~1 and~2]{ADSZ04}: measurable solutions $v$ are constant
almost everywhere on ``recurrent image sets'' and there are plenty 
of such sets by~\cite[Theorem~3(4)]{ADSZ04}, yielding property (2)
of Preposition~\ref{prop-H2}.   Again, 
(H1) can be verified en route to the estimate for $\|R(z)^nv\|$
(the crucial estimate is stated in~\cite[p.~57, line~10]{ADSZ04}
and is a simple consequence of the AFU structure).
\end{proof}

If in addition,
$\mu(y\in Y:\varphi(y)>n)=\ell(n)n^{-\beta}$
is regularly varying with $\beta\in(0,1]$ (which includes the case
when $f$ has good asymptotic expansions near each indifferent fixed point
as in Example~\ref{ex-Thaler}), then again our main results apply.
In particular, for $\beta\in(\frac12,1]$ it follows that for every BV observable $v:X\to\R$ supported on a compact
subset of $X'$,
$\lim_{n\to\infty}w(n)L^nv=\int_Xv\,d\mu$ uniformly on compact subsets of $X'$,
where $w(n)=d_\beta^{-1}m(n)n^{1-\beta}$.
Again, we can consider a larger class of observables as in Theorem~\ref{thm-AFN}.

\begin{pfof}{Theorem~\ref{thm-AFN}}
(a) This is identical to the proof of Theorem~\ref{thm-thaler} with H\"older
replaced by BV.   
By Theorem~\ref{thm-B}, it suffices to observe that
$\|L^kv_k\| \ll \mu(\varphi=k+1)\|u\|_{BV}$.

\vspace{1ex}
\noindent(b)
As in the proof of Proposition~\ref{prop-B}, it suffices to prove the $\ge$ inequality in the $\liminf$ statement.  Again, define 
$v(k)=\sum_{j=0}^k1_{X_j}v$.
Choose $v^-\in\mathcal{B}(X)$ such that $v\ge v^-$ on $\bigcup_{j=0}^k X_j$
and define
$v^-(k)=\sum_{j=0}^k1_{X_j}v^-$.
By Proposition~\ref{prop-B}, for $k$ fixed,
\begin{align*}
\liminf_{n\to\infty} m(n)n^{1-\beta}L^nv & \ge 
\liminf_{n\to\infty} m(n)n^{1-\beta}L^nv^-\ge 
\liminf_{n\to\infty} m(n)n^{1-\beta}L^n v^-(k)
\\ & =d_\beta\int_X v^-(k)\,d\mu.
\end{align*}
Since $v$ is Riemann integrable,
it follows that
$\liminf_{n\to\infty} m(n)n^{1-\beta}L^nv\ge \int_X v(k)\,d\mu$.  Now let $k\to\infty$ to complete the proof.

\vspace{1ex}
\noindent(c)
We have $1_YL^nv\ll 1_YL^n(1/h)
=\sum_{k=0}^n T_{n-k}L^k(1_{X_k}/h)
\ll\sum_{k=0}^n T_{n-k}\tilde L^k1_{X_k}$,
and hence
$\|1_YL^nv\|\ll \sum_{k=0}^n \|T_{n-k}\|\|\tilde L^k1_{X_k}\|\ll [\{\|T_n\|\}\star 
\{\mu(\varphi=n)\}]_n \ll \ell(n)n^{-\beta}$.
\end{pfof}

\begin{rmk}
(i) The proof of Theorem~\ref{thm-AFN}(a) shows that the hypotheses are easily
generalised just as in Remark~\ref{rmk-thaler}.
Suppose that $v\in L^1(X,\mu)$ and that $v$ is BV on each $X_k$.  
Then it suffices that 
$\sum_{k\ge n} \mu(\varphi=k)\|1_{X_k}(vh)\|_{BV(X_k)}=o(w(n)^{-1})$.  Alternatively, using condition~(i) in
Theorem~\ref{thm-B}, it suffices that $\mu(\varphi=n)\|1_{X_n}(vh)\|_{BV(X_n)}$ is summable and $o(n^{-1})$.

\vspace{1ex}
\noindent(ii) For Thaler maps, which are AFN with Gibbs-Markov first return 
maps, we can work with H\"older or BV norms.
%
\end{rmk}

\subsection{Pomeau-Manneville maps}
\label{sec-PM}

In this subsection, we verify that our results on second order asymptotics
are applicable to certain Pomeau-Manneville intermittency maps, in particular
the family~\eqref{eq-LSV} studied by Liverani 
{\em et al.}~\cite{LiveraniSaussolVaienti99}.  
Write the invariant measure as $d\mu=h\,dm$ where $m$ is Lebesgue measure
and $h$ is the density.

\begin{prop} Suppose that $f:X\to X$ is given as in~\eqref{eq-LSV} with
$\beta=1/\alpha\in(0,1]$.  Then $\mu(\varphi>n)= cn^{-\beta}+O(n^{-2\beta})$
where $c=\frac 14\beta^\beta h(\frac12)$.  
\end{prop}

\begin{proof}
First, let $Y=[\frac12,1]$ with partition sets $Y_j=\{\varphi=j\}$.
Let $x_n\in(0,\frac12]$ be the sequence with $x_0=\frac12$ and 
$x_n=fx_{n+1}$, so $x_n$ is decreasing and $x_n\to0$.   
A standard argument shows that $x_n\sim \frac12\beta^\beta n^{-\beta}$
(cf.~\cite[Corollary~1]{Sarig02}).   Then $x_n-x_{n+1}=fx_{n+1}-x_{n+1}
=2^\alpha x^{\alpha+1}=O(n^{-(\beta+1)})$.  Hence
$x_n=\frac12\beta^\beta n^{-\beta}+O(n^{-(\beta+1)})$.
%

Write $Y_n=[y_{n-1},y_{n-2}]$.  Then $f([\frac12,y_n])=[0,x_n]$.  
In particular $m(\varphi>n)=\frac12 m([0,x_{n-1}])
=\frac12 x_{n-1}=m(\varphi>n)=\frac 14\beta^\beta n^{-\beta}
+O(n^{-(\beta+1)})$.

The density $h$ is globally Lipschitz on $(\epsilon,1]$
for any $\epsilon>0$ (see for example~\cite{Hu04} 
or~\cite[Lemma~2.1]{LiveraniSaussolVaienti99}.  Hence
$\mu(\varphi>n)=m(\varphi>n)(h(\frac12)+O(n^{-\beta}))$,
and the result for $Y=[\frac12,1]$ follows.
The same estimates are obtained by inducing on the set $Y=[x_q,1]$ for any fixed
$q\ge0$.
\end{proof}

The next result is immediate by Theorems~\ref{thm-second}
and~\ref{thm-second_one}.
\begin{cor} \label{cor-PM_second}
Suppose that $f:X\to X$ is given as in~\eqref{eq-LSV} with 
$\beta\in(\frac12,1]$.  Suppose that $v:[0,1]\to\R$ 
is H\"older or bounded variation supported on a compact subset of $(0,1]$.
Let $m(n)=c$ if $\beta\in(\frac12,1)$ and $m(n)=c\log n$ if
$\beta=1$.  Let $\gamma=\min\{1-\beta,\beta-\frac12\}$.  Then
\[
m(n)n^{1-\beta}L^nv=d_\beta {\SMALL\int}_Xv\,d\mu + O(m(n)^{-1}n^{-\gamma})
\enspace\text{uniformly on compact subsets of $(0,1]$}.
\]
Moreover, if $\beta\in(\frac34,1]$, then
$\lim_{n\to\infty}m(n)n^{1-\beta}\{m(n)n^{1-\beta}L^nv-d_\beta \int_Xv\,d\mu \}=d_{\beta,1}\int_Xv\,d\mu$, uniformly on compact subsets of $(0,1]$,
where typically $d_{\beta,1}\neq0$.  \qed
\end{cor}

Finally, we mention a result on higher order asymptotics for 
general observables $v(x)=x^q$. 

\begin{thm} \label{thm-PMsecond_q}
Suppose that $f:X\to X$ is given as in~\eqref{eq-LSV} with 
$\beta\in(\frac12,1]$.  
Let $v(x)=x^q$ where $\beta(q+1)>1$.
Then
\begin{align*}
& n^{1-\beta}L^nv=c^{-1}d_\beta {\SMALL\int}_Xv\,d\mu + O(n^{-(1-\beta)},\,n^{-(\beta-\frac12)}, n^{-(\beta(q+1)-1)}\log n), \enspace\text{for $\SMALL\beta\in(\frac12,1)$}, \\
& (\log n)L^nv=c^{-1}{\SMALL\int}_Xv\,d\mu + O((\log n)^{-1}), \enspace\text{for $\beta=1$},
\end{align*}
uniformly on compact subsets of $(0,1]$.
\end{thm}

\begin{proof}
We give the details for $\beta\in(\frac12,1)$.  
By Theorem~\ref{thm-second}, 
$T_n=c^{-1}d_\beta n^{-(1-\beta)}P+O(n^{-2(1-\beta)}P)+O(n^{-\frac12})$.
Following the proof of Theorem~\ref{thm-B} with $v_j=1_{X_j}$ and
$X_j=(x_j,x_{j-1}]$,
\begin{align*}
& \qquad cd_\beta^{-1}n^{1-\beta}L^nv-\int v   =cd_\beta^{-1}n^{1-\beta}\sum_{j=0}^n T_{n-j}L^jv_j-\int v = A-B+C+D, 
\end{align*}
where
\begin{align*}
&  A= n^{1-\beta}\sum_{j=0}^n \{(n-j)^{-(1-\beta)}-n^{-(1-\beta)}\}\int v_j 
, \quad
B= \sum_{j>n}\int v_j, \\
& C= O\Bigl(n^{1-\beta}\sum_{j=0}^n (n-j)^{-2(1-\beta)}\int v_j\Bigr)  ,
\quad 
D= O\Bigl(n^{1-\beta}\sum_{j=0}^n (n-j)^{-\frac12}\|L^jv_j\|\Bigr).
 \end{align*}
Now $h(x)\approx x^{-1/\beta}$, so
$\int v_j\,d\mu\ll \int_{X_j}x^{q-1/\beta}\,dx\ll j^{-(\beta q-1)}\mu(X_j)
\ll j^{-\beta(q+1)}$.  Hence
$B=O(n^{-(\beta(q+1)-1)})$ and $C=O(n^{-(1-\beta)})$.
Using the estimate
\begin{align*}
(n-j)^{-(1-\beta)}-n^{-(1-\beta)} & =
(n-j)^{-(1-\beta)}n^{-(1-\beta)}\{n^{1-\beta}-(n-j)^{1-\beta}\}
\\ & =n^{-(1-\beta)}\{(1+j/(n-j))^{1-\beta}-1\}
\ll n^{-(1-\beta)}(n-j)^{-1}j,
\end{align*}
we obtain that
$A  \ll \sum_{j=0}^n (n-j)^{-1}j\int v_j
\ll \sum_{j=0}^n (n-j)^{-1}j^{-(\beta(q+1)-1)}\ll n^{-(\beta(q+1)-1)}\log n$.

Now $h$ is monotone~\cite[Section~2]{LiveraniSaussolVaienti99} and so
$\|h\|_{BV(X_j)}\approx |1_{X_j}h|_\infty$ for each $j$.  Similarly for $v$ and
we obtain that
$\|vh\|_{BV(X_j)}\le \|v\|_{BV(X_j)} \|h\|_{BV(X_j)}\ll 
|1_{X_j}v|_\infty |1_{X_j}h|_\infty
=|1_{X_j}vh|_\infty \ll j^{-(\beta q-1)}$.
Hence $\|L^jv_j\|\ll \|\tilde L^j(v_jh)\|\ll j^{-(\beta+1)}\|v_jh\|_{BV(X_j)}\ll j^{-\beta(q+1)}$.
It follows that $D=O(n^{-(\beta-\frac12)})$.
%
%
\end{proof}

\paragraph{Acknowledgements}
The research of IM and DT was supported in part by EPSRC Grant EP/F031807/1.
We are very grateful to S\'ebastien Gou\"ezel and Roland Zweim\"uller
for helpful discussions and encouragement, and to the referees for helpful suggestions.

\end{document}